\documentclass{article}
\usepackage{graphicx} 

\usepackage{booktabs}
\usepackage{float}
\usepackage{amsfonts,amssymb,amsmath}
\usepackage{amsmath}
\usepackage{amsthm}
\usepackage{paralist}
\usepackage{subcaption}
\usepackage{float}
\usepackage{url}
\usepackage{xcolor}
\usepackage{bbm} 
\usepackage[round]{natbib}
\usepackage[hidelinks]{hyperref}
\usepackage{algorithm}
\usepackage{algorithmic}

\renewcommand{\le}{\leqslant}
\renewcommand{\ge}{\geqslant}
\renewcommand{\emptyset}{\varnothing}

\renewcommand{\geq}{\geqslant}

\newcommand{\bsr}{\boldsymbol{r}}

\newcommand{\eqd}{\stackrel{\mathrm{d}}{=}}

\newcommand{\bsone}{\boldsymbol{1}}
\newcommand{\phz}{\phantom{0}}

\newcommand{\INPUT}{\item[\textbf{Input:}]}
\newcommand{\OUTPUT}{\item[\textbf{Output:}]}

\newcommand{\bsx}{\boldsymbol{x}}

\newcommand{\bsw}{\boldsymbol{w}}
\newcommand{\bsy}{\boldsymbol{y}}
\newcommand{\bsz}{\boldsymbol{z}}

\newcommand{\bszero}{\boldsymbol{0}}
\newcommand{\diag}{\mathrm{diag}}
\newcommand{\dist}{\mathrm{dist}}
\newcommand{\rqmc}{\mathrm{RQMC}}
\newcommand{\mc}{\mathrm{MC}}
\newcommand{\real}{\mathbb{R}}
\newcommand{\natu}{\mathbb{N}}
\newcommand{\tran}{\mathsf{T}}

\newcommand{\runif}{\mathbf{U}} 
\newcommand{\rd}{\,\mathrm{d}}
\newcommand{\vol}{\mathrm{vol}}

\newcommand{\giv}{\!\mid\!}
\newcommand{\lattice}{\mathrm{Lat}}
\newcommand{\net}{\mathrm{Net}}

\newcommand{\e}{\mathbb{E}}
\newcommand{\var}{\mathrm{var}}

\newcommand{\simiid}{\stackrel{\mathrm{iid}}{\sim}}

\newcommand{\olt}{\overline\tau}
\newcommand{\ult}{\underline\tau}
\newcommand{\ols}{\overline S}
\newcommand{\uls}{\underline S}

\newcommand{\cx}{\mathcal{X}}

\newtheorem{prop}{Proposition}

\title{Walk on spheres and Array-RQMC}
\author{Valérie N. P. Ho\\Stanford University \and Art B. Owen\\Stanford University}
\date{July 2026}

\begin{document}

\maketitle
\begin{abstract}
We use Array-RQMC sampling in a walk on spheres (WoS) algorithm for Dirichlet boundary value problems.  On a collection of problems, we find that Array-RQMC-WoS reduces the Monte Carlo MSE or variance by factors ranging from $71$-fold to $3087$-fold at $n=2^{17}$ trajectories.  The variance is known to be $o(1/n)$ but attains empirical rates between $n^{-1.4}$ and $n^{-1.8}$ in our examples. A simpler RQMC-WoS algorithm studied in \cite{rqmcwos:tr} has more theoretical support but only reduced variance by 1.8 to 10.7-fold on the same set of examples. In order to explain this improvement, we introduce a column-wise mean dimension of the RQMC error based on Sobol' indices. It matches the usual mean dimension for Monte Carlo and the mean dimension of a dual lattice error for randomized lattices. We find for a gasket example from \cite{wosoneweekend} that the mean dimension of Array-RQMC-WoS errors is much higher than an analogous Array-MC-WoS algorithm has.
\end{abstract}

\section{Introduction}

The walk on spheres (WoS) algorithm is a grid-free Monte Carlo method for solving a boundary value problem (BVP). In the simplest setting, a function $u$ has Laplacian $\Delta u=0$ in the interior of a bounded closed set $\Omega\subset\real^d$ and has known boundary values $u(\bsz)=b(\bsz)$ for $\bsz\in\partial\Omega$. The goal is to compute $u(\bsz_0)$ at some point $\bsz_0\in\Omega^\circ$.  More general problems have $\Delta u(\bsz)=g(\bsz)$ for a nonzero source function $g$. The boundary condition we just described is the Dirichlet condition.  There are other boundary conditions of interest but this paper does not consider them.

When $g\equiv0$, it is known that $u(\bsz_0)$ is the expected value of $b(\bsz_\tau)$ where $\bsz_t$ for $t\in[0,\infty)$ is Brownian motion in $\real^d$ starting at $\bsz_0$ and $\tau =\inf\{t\geq0\mid \bsz_t\in\partial\Omega\}$.  When that Brownian motion first reaches the sphere $S(\bsz_0,r)=\{\bsz\in\real^d\mid \Vert\bsz-\bsz_0\Vert =r\}$,  it does so with a uniform distribution on that sphere.  Instead of simulating spatial Brownian motion directly, we can simply take $\bsz_1\sim\runif(S(\bsz_0,r_1))$ where $r_1=\dist(\bsz_0,\partial\Omega)$ is the radius of the largest sphere centered on $\bsz_0$ that is contained within $\Omega$.  The walk continues taking steps $\bsz_k\sim\runif(S(\bsz_{k-1},r_k))$ for $r_k=\dist(\bsz_{k-1},\partial\Omega)$. The stopping rule is to terminate the walk at step $\tau =\min\{ k\ge1\mid \dist(\bsz_k,\partial\Omega)< \varepsilon\}$. Then we project $\bsz_\tau$ onto $\partial\Omega$ finding a point $\bar\bsz_\tau$ that is any minimizer of $\dist(\bsz_\tau,\bsz)$ over $\bsz\in\partial\Omega$ and return the value $u(\bar\bsz_\tau)=b(\bar\bsz_\tau)$.  A Monte Carlo algorithm repeats this process $n$ times independently and returns $\hat u(\bsz_0) = (1/n)\sum_{i=1}^n b(\bar\bsz_{i,\tau(i)})$.  When there is a nonzero source term $g$, the WoS algorithm also makes use of function values $g(\bsw_k)$ for $\bsw_k$ uniformly distributed in the ball of radius $r_k$ centered at $\bsz_{k-1}$ as described in Section \ref{sec:wos} below.

The WoS algorithm originates with \cite{mull:1956}.  Its applicability has been greatly expanded by Sawhney and Crane and others, in a series of papers starting with \cite{sawh:cran:2020} that consider large scale computations and many different BVPs.  A quasi-Monte Carlo (QMC) version of WoS was given by \cite{masc:kara:hwan:2004}. The accuracy of QMC is difficult to study because the integrands in WoS are typically supported in some set $\Theta\subset[0,1)^k$ that does not have axis-aligned boundaries.  The integrands then have infinite variation in the sense of Hardy and Krause, which makes the Koksma-Hlawka inequality uninformative.  

A randomized quasi-Monte Carlo (RQMC) version of WoS was studied in \cite{rqmcwos:tr}.  In scrambled net RQMC, we only need the integrand to be in $L_2$ to establish that the variance is $o(1/n)$ as $n\to\infty$, and for finite $n$, the variance is also bounded by some constant $\Gamma$ times the Monte Carlo (MC) variance.  Moreover, for a set $\Theta\subset[0,1)^k$ where $\partial\Theta$ has finite $k-1$ dimensional Minkowski content, \cite{he:wang:2015} establish an RQMC variance of $O(n^{-1-1/k})$ for integration of the indicator function $\bsone_{\Theta}$. Under that same Minkowski content condition, \cite{liu:2025} establishes the same rate for integration of $g\times\bsone_{\Theta}$ when $g$ obeys a mild boundary growth condition.  \cite{rqmcwos:tr} use these results to show that certain RQMC-WoS integrals have errors that decay as $O(n^{-1-1/k})$. In \cite{rqmcwos:tr}, we estimated variance reduction factors ranging from $1.8$ to $10.7$ fold at $n=2^{17}$ in a range of problems based on some regressions of error versus $n$ on the log-log scale. We also saw empirical convergence rates better than the MC rate, following $O(n^{-1-\delta})$ for small $\delta>0$. Those rates were supported by conditions on $\Omega$ sufficient to control the $k-1$ dimensional Minkowski content of $\partial\Theta$ where the WoS problem was to estimate $\Pr(\bsx\in\Theta)$ under $\bsx\sim\runif([0,1)^k)$.

In this paper, we replace RQMC sampling by Array-RQMC sampling of \cite{lecu:leco:tuff:2008}. We use the Hilbert space-filling curve proposed by \cite{gerber2015sequential} for sequential QMC to align the Array-RQMC steps. The resulting Array-RQMC-WoS algorithms yield much better convergence rates than the RQMC-WoS algorithms in \cite{rqmcwos:tr}. Empirical variance rates better than $O(n^{-1.5})$ are very common as are extremely large variance reduction factors compared to those of RQMC-WoS.  The variance reductions are large enough that the bias from $\varepsilon>0$ can be detected in some examples. That bias is commonly ignored in WoS problems. Past empirical success of Array-RQMC goes well beyond the known theoretical explanations and that remains true for Array-RQMC-WoS in this setting where even RQMC-WoS is difficult to study.

This paper is organized as follows.  Section~\ref{sec:notation} introduces further notation, and gives background on RQMC, Array-RQMC,  Hilbert curves and defines our Array-RQMC-WoS algorithm.  Section~\ref{sec:numerical}  presents some empirical results where Array-RQMC-WoS brings very large variance reductions. {The code used for those numerical examples is publicly available at \url{https://github.com/hoval58/RQMC-WoS}. Section~\ref{sec:variants} considers variants of the algorithm including the use of Fibonacci lattices, alternative ways to handle converged/inactive walks, and algorithms that take a fixed number of steps. Section~\ref{sec:explanations} investigates an effective dimension explanation of the accuracy ranking where Array-RQMC-WoS is better than RQMC-WoS which is better than MC-WoS. That explanation uses a column-wise effective dimension measure. It also compares the discrepancy of converged points when $\Omega$ is the unit disk and $\bsz_0=(t,0)^\tran$.  In this setting, the true distribution is known and the Array-RQMC-WoS points have lower discrepancy than the RQMC-WoS points which in turn have lower discrepancy than the MC-WoS points. Section~\ref{sec:discussion} has some concluding remarks.

\section{Background and notation}\label{sec:notation}

Here we provide more details about WoS as well as QMC, RQMC, Hilbert space-filling curves, and Array-RQMC. This allows us to define RQMC-WoS and Array-RQMC-WoS and to summarize known results.

We give some notation here. Additional notation is introduced where it is used.
We use $\natu$ for the set of positive integers and for a set $u$ we use $|u|$ for its cardinality. For random variables $A$ and $B$, $A\eqd B$ means that they have the same distribution.

We use $\Vert\cdot\Vert$ to denote Euclidean distance.  For a point $\bsz\in\real^d$ and a non-empty set $\Theta\subset\real^d$ we use
$\dist(\bsz,\Theta) = \inf\{\Vert\tilde\bsz-\bsz\Vert\mid \tilde\bsz\in\Theta\}$. For a closed set $\Theta$, the infimum is attained at some not necessarily unique point $\bar \bsz$. We use $\bsone_\Theta$ for the function taking the value $1$ on the set $\Theta$ and $0$ on its complement.

For $\bsz\in\real^d$ and $r\ge0$, we let 
$B(\bsz,r)=\{\tilde\bsz\in\real^d\mid \Vert\tilde\bsz-\bsz\Vert\le r\}$
and
$S(\bsz,r)=\{\tilde\bsz\in\real^d\mid \Vert\tilde\bsz-\bsz\Vert=r\}$ denote
balls and spheres, respectively. We use $B_d$ and $S_d$ when we want to emphasize the dimension. The Lebesgue measure on $\real^d$ is denoted by $\lambda_d$.

The Laplacian of a twice differentiable function $u:\real^d\to\real$ is denoted $\Delta u = \sum_{j=1}^d\partial^2u/\partial z_j^2$. This is the positive semi-definite Laplacian. Some authors use $\Delta u$ to denote a negative semi-definite Laplacian equal to our $-\Delta u$.

\subsection{WoS}\label{sec:wos}

We use WoS to solve Dirichlet BVPs.  There, we have a set $\Omega\subset\real^d$ which is the closure of a bounded open domain.  We want to solve for $u(\bsz_0)$ where $\bsz_0\in\Omega^\circ$. We are given
\begin{equation}\label{eq:Dirichlet}
\begin{aligned}
    \Delta u(\bsz) &= g(\bsz), && \bsz \in \Omega, \quad\text{and}\\
    u(\bsz) &= b(\bsz), && \bsz \in \partial\Omega,
\end{aligned}
\end{equation}
for a source function $g$ and a boundary function $b$.
Replacing $\Delta u$ by the negative semi-definite Laplacian is equivalent to using the usual Laplacian with $-g$ as the source.

When $g\equiv 0$, the solution to \eqref{eq:Dirichlet} is given by $u(\bsz_0)=\e( b(\bsz_\tau))$ where $\bsz_t$ is $d$-dimensional Brownian motion starting from $\bsz_0$ at time $t=0$ and $\tau = \inf\{t\ge0\mid \bsz_t\in\partial\Omega\}$. As described in the introduction, we iterate $\bsz_k=\bsz_{k-1}+\runif(S(\bszero,r_k))$. 
It is called a \textit{walk on spheres} because it steps from sphere to sphere within $\Omega$.

The WoS algorithm will ordinarily never reach $\partial\Omega$. We stop the algorithm at step $\tau = \min\{k\in\natu\mid \dist(\bsz_k,\partial\Omega)<\varepsilon\}$.  Then
we let $\bar\bsz_\tau$ be the projection of $\bsz_\tau$ on the closed set $\partial\Omega$ with an arbitrary choice when that projection is not unique. We return the value $b(\bar\bsz_\tau)$ as an almost unbiased estimate of $u(\bsz_0)$.  For the cases we study, the number of steps to the boundary is $O(\log(1/\varepsilon))$ using a result from \cite{bind:brav:2012}. In this sense, the boundary `attracts' the walk. The bias from stopping short of $\partial\Omega$ is less well understood but extensive computations in \cite{masc:hwan:2003} show that it is regularly $O(\varepsilon)$.  It is then recommended to use small $\varepsilon$ in WoS.

A full Monte Carlo estimate is
$$
\hat\mu_{\mc}(\bsz_0)=\frac1n\sum_{i=1}^nb(\bar\bsz_{i,\tau(i)})
$$
where the $i$th sampled trajectory has components $\bsz_{ik}$ for $0\le k\le \tau(i)$ with all $\bsz_{i0}=\bsz_0$ and $\tau(i)$ is the number of steps it takes for the $i$th walk to enter $\partial\Omega_\varepsilon = \{\bsz\in\real^d\mid \dist(\bsz,\partial\Omega)<\varepsilon\}$.

When the source function $g$ is nonzero, the WoS algorithm is modified to include some function evaluations of $g$ along the sample path. We present the algorithm from \cite{rqmcwos:tr} which is derived using the account in \cite{sawh:mill:gkio:cran:2023}. We use the following Dirichlet Green's functions for two and three dimensional balls,
$$
G_2^{B(\bsz,r)}(\bsz,\bsw)=
\frac{\log(r\Vert\bsw-\bsz\Vert^{-1})}{2\pi}
\quad\text{and}\quad
G_3^{B(\bsz,r)}(\bsz,\bsw)=
\frac1{4\pi}\Bigl(\frac1{\Vert\bsw-\bsz\Vert}-\frac1r\Bigr).
$$
At the same time that we generate $\bsz_k\sim\runif(S(\bsz_{k-1},r_k))$, we also sample $\bsw_k\sim\runif(B(\bsz_{k-1},r_k))$. Then, the estimate is
\begin{align}\label{eq:woswithsource2}
\hat u(\bsz_0)
=
b(\bar{\bsz}_{\tau})
-\sum_{k=1}^{\tau}
\vol(B(\bsz_{k-1},r_k))\,
G_d^{B(\bsz_{k-1},r_k)}(\bsz_{k-1},\bsw_k)\, g(\bsw_k)
\end{align}
that we average over $n$ independent replicates using independent vectors 
\begin{align*}
\bsz_{k}&\sim \runif(S(\bsz_{k-1},r_k))\eqd \bsz_{k-1} +r_k\runif(S(\bszero,1))\quad\text{and}\\
\bsw_k&\sim \runif(B(\bsz_{k-1},r_k))\eqd\bsz_{k-1}+r_k\runif(B(\bszero,1))
\end{align*}
each time.

In MC sampling, we use transformations $\psi_0=\psi_{0,d}:[0,1)^{s_0}\to S_d(\bszero,1)$ and $\psi_1=\psi_{1,d}:[0,1)^{s_1}\to B_d(\bszero,1)$ where $\psi_{0,d}(\bsx)\sim\runif(S_d(\bszero,1))$ for $\bsx\sim\runif([0,1)^{s_0})$ and $\psi_{1,d}(\bsx)\sim\runif(B_d(\bszero,1))$ for $\bsx\sim\runif([0,1)^{s_1})$.  To run WoS for $K$ steps, we may use $K(s_0+s_1)$ uniform variables when there is a source term and $Ks_0$ uniform variables when there is no source term. Standard algorithms have $s_0=d-1$ and $s_1=d$.  Our two and three dimensional sphere samplers are $\psi_{0,2}(x)=\theta(x):=(\cos(2\pi x),\sin(2\pi x))$ for $x\in[0,1)$ and $\psi_{0,3}:[0,1)^2\to S_3(\bszero,1)$ from the hat-box theorem as described in \cite{rqmcwos:tr}. We use $\psi_{1,2}(\bsx)=(\sqrt{x_1}\cos(2\pi x_2),\sqrt{x_1}\sin(2\pi x_2))$ for $\bsx\sim\runif([0,1)^2)$ to sample $\runif(B_2(\bszero,1))$. None of our examples require samples from $B_3(\bszero,1)$.

\subsection{RQMC sampling}\label{sec:sampling}

We assume some familiarity with QMC and RQMC.  For background, see \cite{nied:1992}, \cite{dick:pill:2010} or \cite{practicalqmc}. As in MC, we use points $\bsx_i\in[0,1)^s$ to estimate $\int_{[0,1)^s}f(\bsx)\rd\bsx$ by $(1/n)\sum_{i=1}^nf(\bsx_i)$. In RQMC, each $\bsx_i\sim\runif([0,1)^s)$ individually while collectively, the points have low discrepancy, meaning that the discrete uniform distribution on those $n$ points is close to the continuous uniform distribution on $[0,1)^s$. Plain QMC uses deterministic points with the second property. To run WoS for $K$ steps, we need $K(s_0+s_1)=K(2d-1)$ uniform variables when there is a source term and $Ks_0=K(d-1)$ uniform variables when there is no source term. 

We use QMCPy \citep{QMCPy} to construct some of our RQMC points.  The RQMC computations in \cite{rqmcwos:tr} showed very nearly equal performance for multiple RQMC constructions. In this paper, we consider only two constructions. The first consists of Sobol' points \citep{sobo:1967:russ} with direction numbers from \cite{joe:kuo:2008} and the matrix scramble of \cite{mato:1998:2}.  This is one of the most commonly considered RQMC algorithms. 

The other major branch of RQMC consists of randomly shifted lattice rules. Lattice rules are interesting for WoS because for $d=2$, the sample points $\bsz_k$ are periodic functions of the $k$ uniform variables used to generate the walk.  Periodicity favors lattice rules \citep{sloa:joe}, although the full benefit of lattice sampling also requires smoothness.  Following \cite{l2018sorting}, we use Korobov lattice rules. These are points of the form $\bsx_i = i\bsz/n \bmod 1$ for $i=1,\dots,n$ with $\bsz=(1,a,a^2 \mod n,\dots,a^{s} \mod n)$ for a carefully chosen integer $a$ in $\{2,\dots,n-1\}$, with $\gcd(a,n)=1$. To pick $a$, we use \textit{LatNet Builder} from \cite{lecuyer2022latnet}, restricting the search to Korobov generating vectors and minimizing the $\mathcal{P}_2$ criterion for each value of $n$. Finally, we randomize the lattice rule by generating a single random shift $\Delta\sim\runif([0,1)^s)$ and adding it modulo $1$, coordinatewise, to every point. Note that in \cite{rqmcwos:tr}, we used the default rank-1 lattice rule implemented in QMCPy, which is based on generating vectors from Frances Kuo, not customized to specific values of $n$. In most of our examples, these lattices performed slightly worse than the Korobov lattices when used within the Array-RQMC algorithm.
}

RQMC-WoS \citep{rqmcwos:tr} generates $n$ trajectories through $\Omega$. They all start at $\bsz_0$. It uses $n$ RQMC points $\bsx_i\in[0,1)^s$ to advance $\bsz_{1,k-1},\dots,\bsz_{n,k-1}$ to  $\bsz_{1k},\dots,\bsz_{nk}$ and if necessary sample the source function.  The update for $\bsz_{i,k-1}$ uses the point $\bsx_i$. For converged points $\bsz_{ik}\in\partial\Omega_\varepsilon:=\{\bsz\in\real^d\mid \dist(\bsz,\partial\Omega)<\varepsilon\}$, we do not have to compute $\bsz_{i,k+1}$.  To handle $K$ steps, we use an RQMC point set in $[0,1)^{Ks}$ and if any points have not converged by step $K$, then we use $\bar \bsz_{iK}$.

The RQMC point sets we use all have the following stratification property: for each $j=1,\dots,s$, each interval
$[c/n,(c+1)/n)$, where $c=0,1,\dots,n-1$, contains exactly one of
$x_{1j},\dots,x_{nj}$.

\subsection{Array-RQMC}

An RQMC algorithm to sample $K$ steps of a stochastic process like WoS requires $Ks$ uniform variables, that is, we need to find a matrix in $[0,1)^{n\times Ks}$ to sample all of our trajectories.  QMC and RQMC performance can degrade in high dimensions. Array-QMC and Array-RQMC use lower dimensional point sets. An up-to-date account of Array-RQMC is in \cite{l2018sorting} along with empirical results.

\cite{leco:tuff:2004} introduced an Array-QMC algorithm for a Markov chain $z_t\in\real$ when that chain updates via $z_{t+1} \gets \phi( z_t,u_{t+1})$ for $u_{t+1}\sim \runif(0,1)$. To advance $n$ walkers, they used a QMC point set $\bsx_i\in[0,1)^2$ at each step. They sort the walkers so that $z_{(1)k}\le z_{(2)k}\le \cdots \le z_{(n)k}$ and also sort the QMC points $\bsx_i$ so that $\bsx_{(1)1}\le \bsx_{(2)1}\le\cdots\le \bsx_{(n)1}$.  Then they advance $z_{(i)k}$ to $z_{(i),k+1}=\phi(z_{(i)k},\bsx_{(i)2})$ (which are then sorted before taking the next step).

When the sample distribution of $\bsx_1,\dots,\bsx_n$ is nearly $\runif([0,1)^2)$, it means that the variables $\bsx_{(i)2}$ used in the updates are nearly independent of the walker locations $z_{(i)k}$, which is a desirable property. \cite{leco:tuff:2004} give conditions under which the sample distribution of the walkers at step $k$ approaches the true $k$-step distribution as $n\to\infty$. Many QMC point sets are constructed with a first column that is already sorted, commonly taking values $(i-1)/n$.  Then, the $d=1$ version of Array-QMC only needs to apply the second column after sorting the walkers. 

\cite{lecu:leco:tuff:2008} present an Array-RQMC algorithm using $\cx_k\in\real^{n\times (s+1)}$ to update $n$ walkers in $\real^d$ at step $k$ where $s$ need not equal $d$.  Each $\cx_k$ is a randomization of a common QMC point set $\cx$. The first column of $\cx_k$ is aligned with the values $h(\bsz_{ik})$ for some sorting function $h:\real^d\to\real\cup\{\infty\}$. We use $h(\bsz_{ik})=\infty$ to mark walks that have converged into $\partial\Omega_\varepsilon$ and do not need further updates. Their objective is to estimate $\mu = \e( \sum_{k=0}^\tau c_k(\bsz_k))$ where $\tau$ is a stopping time and $c_k$ is a cost function.  They give conditions under which the Array-RQMC estimator of $\mu$ is unbiased and, for $d=1$, has variance $O(n^{-3/2})$. When, as is common, the first column has the Latin hypercube property, one sorts the $\bsz_{ik}$ according to $h(\bsz_{ik})$ and then uses the last $s$ columns of $\cx$ to advance the walkers. Some alternative ways to handle converged samplers are discussed in Section~\ref{sec:variants}.

We can incorporate WoS into their framework as follows.  If there is a source function we expand the state space using elements $\bsy_k=(\bsz_k,\bsw_{k+1})$. The update to $\bsy_{k+1}$ depends on $\bsy_k$ only through $\bsz_k$. Then the cost function is
$$c(\bsy_k)= \begin{cases}
    -\vol(B(\bsz_{k},r_{k+1}))\,G_d^{B(\bsz_{k},r_{k+1})}(\bsz_{k},\bsw_{k+1})\, g(\bsw_{k+1}),  &\textrm{if } k<\tau\\
    b(\bar{\bsz}_{\tau}), \quad&\textrm{if } k=\tau\\
    0, &\textrm{if } k>\tau.
\end{cases}$$
The algorithm detects $k\ge\tau$ by $\bsz\in\partial\Omega_\varepsilon$. To distinguish $k>\tau$ from $k=\tau$ we can adjoin a binary convergence flag to the state space.

For $d>1$, choosing an effective sorting function is less straightforward.  The most successful method to date uses a Hilbert space-filling curve to generate a one dimensional ordering of the walkers at step $k$. This was used by \cite{gerber2015sequential} for a particle filtering algorithm that generalizes Array-RQMC.  We describe Hilbert curves in the next subsection before going on to describe Array-RQMC-WoS. Figure \ref{fig:array-rqmc-chart} illustrates the idea of Array-RQMC in a WoS setting using the Hilbert sort.

\begin{figure}
    \centering
    \includegraphics[width=1.1\linewidth]{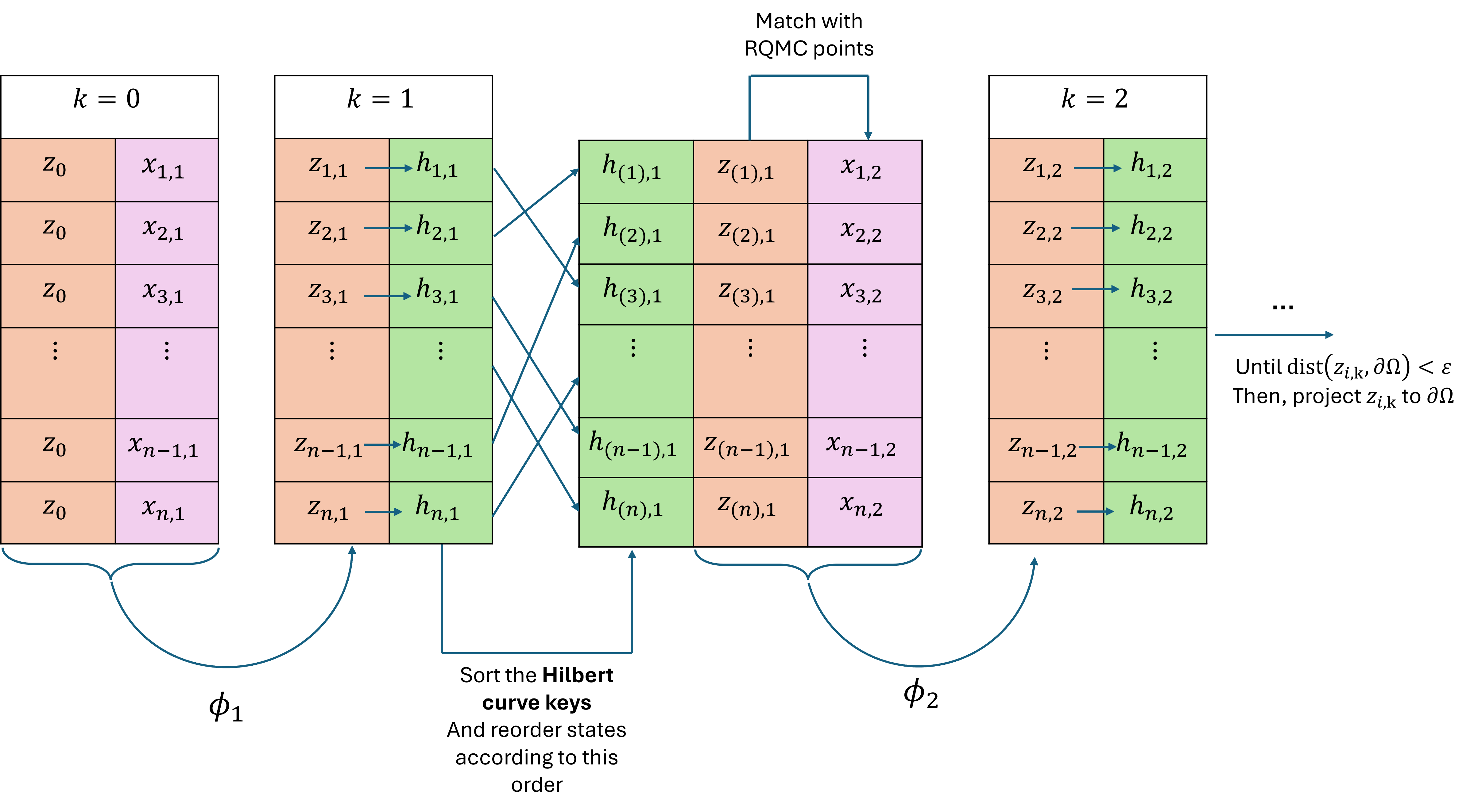}
    \caption{Array-RQMC algorithm up to step $k=2$. At each step, the previous states are mapped through a Hilbert curve to $[0,1)$ and then sorted accordingly, before advancing the chain.
}
    \label{fig:array-rqmc-chart}
\end{figure}

\subsection{Hilbert curves}

The Hilbert curve is a continuous function $H_d:[0,1]\to[0,1]^d$ for $d\ge2$. It is defined as the limit of a sequence of curves, the first four of which are illustrated in Figure~\ref{fig:hilbert}. For background on the construction and properties of Hilbert curves, see Chapter $2$ of \cite{sagan:1994}. For recent work using Hilbert and other space-filling curves in computer graphics see \cite{kell:wach:bind:2022}.
We note some properties of $H_d$ that are given in \cite{he:owen:2015}: 
\begin{compactenum}[\quad\bf1)]
\item If $A\subset[0,1]$ is measurable then $\lambda_d(H_d(A))=\lambda_1(A)$.
\item $H_d$ is surjective but not injective.
\item $\Vert H_d(x)-H_d(\tilde x)\Vert\le 2\sqrt{d+3}|x-\tilde x|^{1/d}$.
\end{compactenum}

We would like our sorting function $h:[0,1]^d\to[0,1]$ to be $H_d^{-1}$. From point 2, there exist distinct $x,\tilde x\in[0,1]$ with $H_d(x)= H_d(\tilde x)$, so $H_d^{-1}$ is not well defined. However, \cite{gerber2015sequential} show that the set of points $\bsz\in[0,1]^d$ with a non-unique pre-image under $H_d$ has measure zero. They also provide a pseudo-inverse $h$ with the property that $H_d(h(\bsz))=\bsz$ for $\bsz\in[0,1]^d$. 

An important reason for using Hilbert curves in this setting is their
locality-preserving behavior from point 3 above: nearby intervals in $[0,1]$ are mapped to
nearby regions of $[0,1]^d$. However, that does not make $h$ locality preserving. While there is a tendency for nearby $\bsz$ and $\tilde\bsz$ to have nearby $h(\bsz)$ and $h(\tilde\bsz)$, the pseudo-inverse $h$ is discontinuous and a region in $[0,1]^d$ can be split into numerous separated subintervals of $[0,1]$.

Because the Hilbert curve is only H\"older $1/d$ in $d$ dimensions, \cite{he:owen:2015} anticipate a diminishing advantage for the method in \cite{gerber2015sequential} as $d$ increases. They also remark that many important problems have dimensions $2$ or $3$ and we have seen that the BVPs motivating WoS often have small $d$.

\begin{figure}[t]
    \centering
    \includegraphics[width=1\linewidth]{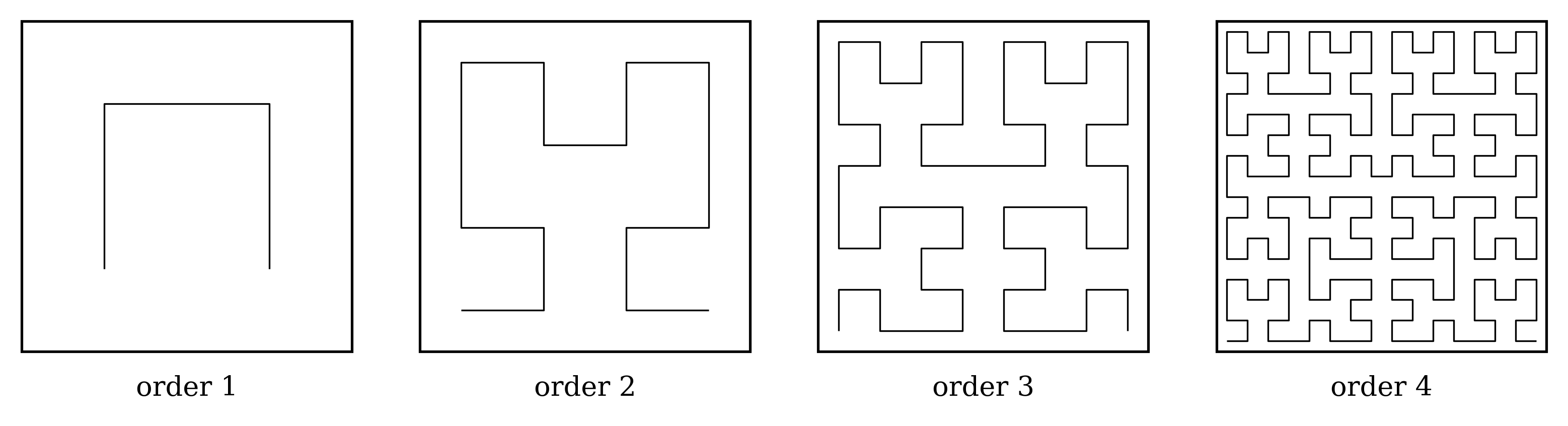}
    \caption{First four approximations of the Hilbert curve in $[0,1]^2$.}
    \label{fig:hilbert}
\end{figure}

\subsection{Array-RQMC-WoS}\label{sec:arrayrqmcwos}

The combination of WoS and Array-RQMC is now straightforward. We start with $\bsz_{i0}=\bsz_0$ for $i=1,\dots,n$. 
Given an RQMC point set $\cx_1\in[0,1)^{n\times s}$ we use the first $s_0$ elements in row $i$ to advance $\bsz_{i0}$ to $\bsz_{i1}$. If there is a source function we use the last $s_1=s-s_0$ elements in row $i$ to sample $\bsw_{i1}$. 

For $k>1$, we sort the points $\bsz_{i,k-1}$ so that the values $h(\bsz_{i,k-1})$ are in increasing order.  Then we use an RQMC point set $\cx_k\in[0,1)^{n\times s}$ to advance the points.  This point set has the last $s$ columns of a larger (implicit) point set that has been permuted if necessary to have its first  column in increasing order as described in \cite{l2018sorting}. The first $s_0$ columns of $\cx_k$ advance the walkers and the last $s_1$ columns are used to evaluate the Green's function. The Array-RQMC-WoS algorithm pseudocode is given in Algorithm \ref{alg:array_rqmc-wos}, for the case where there is no source term. As described in Section~\ref{sec:numerical}, the states are first mapped linearly into $[0,1]^d$ before applying the Hilbert pseudo-inverse. To keep the notation simple in the pseudocode, we suppress this mapping and write $h(\bsz_{ik})$ for the Hilbert key of the transformed state.

\begin{algorithm}  
\caption{Array-RQMC-WoS with Hilbert sorting\label{alg:array_rqmc-wos}}
\begin{algorithmic}[1]

\INPUT
\begin{tabular}[t]{@{}l l@{}}
- & Initial point $\bsz_0\in\Omega$, termination threshold $\varepsilon>0$, number of chains $n$,\\
& maximum number of steps $K$;\\
- & a QMC point set $\widetilde \cx_n\subset[0,1)^{n\times s}$; \\
- & a map $\psi_0:[0,1)^s\to S_d(\bszero,1)$ (uniform direction);\\
- & distance function $r(\bsz)=\dist(\bsz,\partial\Omega)$ and projection $\Pi:\Omega\to\partial\Omega$;\\
- & boundary function $b:\partial\Omega\to\mathbb{R}$;\\
- & sorting function $h:[0,1]^d\to [0,1]$ 
\end{tabular}
\OUTPUT Estimator $\widehat u_n(\bsz_0)$ of $u(\bsz_0)$.

\STATE Set $\bsz_{i0}\leftarrow\bsz_0$ for $i=1,\dots,n$  
\STATE Set $Y_i\leftarrow 0$ for all $i$ \COMMENT{\textit{terminal values}}
\STATE Set $A_i\leftarrow 1$ for all $i$ \COMMENT{\textit{initialize all chains as active}}
\STATE Set $k\leftarrow 1$ \COMMENT{\textit{step counter}} 
\WHILE{$\max_{1\le i\le n} A_i = 1$ and $k\le K$}
    \STATE \COMMENT{\textit{at least one chain is active and the step cap has not been exceeded}}
    \STATE Randomize $\widetilde \cx_n$ afresh into $\cx_{nk}=\{\bsx_{1k},\dots,\bsx_{nk}\}$, independent of previous randomizations.

    \FOR{$i=1,\dots,n$} 
        \IF{$A_i=1$}
            \STATE $r_{ik}\leftarrow r(\bsz_{i,k-1})$
            \STATE $\bsz_{ik}\leftarrow \bsz_{i,k-1}$
            \IF{$r_{ik}\ge \varepsilon$}
                \STATE $\bsz_{ik}\leftarrow \bsz_{i,k-1}+r_{ik}\,\psi_0(\bsx_{ik})$ 
                \COMMENT{\textit{take a WoS step}}
            \ENDIF
        
            \IF{$r_{ik}<\varepsilon$ or $k=K$}
                \STATE $\bar{\bsz}_i \leftarrow \Pi(\bsz_{ik})$; 
                \quad $Y_i \leftarrow b(\bar{\bsz}_i)$; 
                \quad $A_i\leftarrow 0$
                \STATE $\bsz_{ik}\leftarrow \bar{\bsz}_i$
                \STATE $q_{ik}\leftarrow +\infty$ 
                \COMMENT{\textit{push inactive states to the end}}
            \ELSE
                \STATE $q_{ik}\leftarrow h(\bsz_{ik})$ 
                \COMMENT{\textit{compute Hilbert key}}
            \ENDIF
        \ELSE
            \STATE $\bsz_{ik}\leftarrow \bsz_{i,k-1}$; \quad $q_{ik}\leftarrow +\infty$
        \ENDIF
    \ENDFOR

    \STATE Sort the keys $q_{1k},\dots,q_{nk}$ in increasing order, and apply the same permutation to the triples $(\bsz_{ik},A_i,Y_i)$.
    \STATE $k\leftarrow k+1$
\ENDWHILE
\RETURN $\displaystyle \widehat u_n(\bsz_0)\leftarrow \frac{1}{n}\sum_{i=1}^{n} Y_i$
\end{algorithmic}
\end{algorithm}

\subsection{Convergence of the Array-RQMC algorithm}

The best proven rate for the MSE of Array-RQMC is $o(n^{-1})$ when using the Hilbert sort, as shown in Theorem $7$ of \cite{gerber2015sequential}. They study Array-RQMC as a special case of sequential QMC in which certain importance weights have been set to zero. Their method of proof is to show that a low discrepancy property holding at step $k$ implies that property will be preserved into step $k+1$.  This strategy was also used for Array-QMC in dimension $1$ by \cite{leco:tuff:2004}.

Sufficient conditions for a variance of $O(n^{-3/2})$ for Array-RQMC appear in \cite{lecu:leco:tuff:2008} 
for a $d=1$ dimensional walk. 
Their Assumption 2 requires $n$ to be a square and the RQMC point set must have one point independently sampled uniformly within each of $n$ square cells of size $1/\sqrt{n}\times1/\sqrt{n}$. Common RQMC point sets are not quite of this form though some randomized orthogonal arrays \citep{roavar1} have that structure.

The WoS problems do not have $d=1$ because then $\Omega$ would be an interval and simpler algorithms could solve the BVP.  The Hilbert sort projects $\Omega$ into $[0,1]$ allowing us to view $d$-dimensional  points $\bsz_{ik}$ as one dimensional points $h(\bsz_{ik})$. Then, we can use $2$-dimensional point sets (when there is no source function) to advance the walkers just as \cite{leco:tuff:2004} and \cite{l2018sorting} do. Assumption 1 in \cite{lecu:leco:tuff:2008} requires the point $z_{ik}$ to be a monotone function of $x_{ik}\in[0,1)$ for each $z_{i,k-1}$, as occurs when the next state is generated by inversion
from a single uniform random variable. This monotonicity does not hold for a projected walk because $h( \bsz_{i,k-1}+r_{ik}\theta(x))$ is not monotone in $x\in[0,1)$.

Empirical results in \cite{l2018sorting} show much better rates in $n$ than have been proved. In a one dimensional problem, the mean squared discrepancy decayed at rates between $O(n^{-1.50})$ and $O(n^{-1.92})$. The variance of estimates of $\e(Y)$ and $\e(Y^2)$ decayed even faster, and in an Asian option problem using the Hilbert sort, the observed log-log slopes ranged from $-1.54$ to $-2.03$.

When there is a source function we could consider the state space to be $\Omega^2$ with entries $(\bsz_k,\bsw_{k+1})$.  That would require using a Hilbert sort in $[0,1]^{2d}$ which is only H\"older $1/(2d)$, so we have only used $\bsz_k$ in our Hilbert sort.

\section{Numerical experiments}\label{sec:numerical}

All of our experiments have bounded $\Omega$ that we enclose in a hyper-rectangular region which is then mapped linearly to the domain $[0,1]^d$ of the Hilbert sorting function.  Four of the examples have $\Omega\subset[-1,1]^d$, and for them, this mapping is $\eta(\bsz)=(\bsz+1)/2$ componentwise. Our number $n$ of walkers is always a power of two, ranging from $4$ to $2^{17}=131{,}072$.

For our examples, we graph the sampling variance versus $n$. When the true solution is known, we also present an MSE versus $n$.  When plotting these values we attach reference lines.
Those are always fit by a regression of log variance (or MSE) on $\log(n)$ for sample sizes $n\ge 2^7$, but graphed for all sample sizes.
If we do not specify whether the slopes are for lattice or Sobol' sampling, then data for both of those methods has been pooled.  

In \cite{rqmcwos:tr} the RQMC-WoS variances closely followed reference lines and we estimated variance reduction factors from those regressions. In this paper, we see and interpret departures of Array-RQMC-WoS from the reference lines and report estimated variance reduction factors from replicates at specific values of $n$. To get comparable numbers, we do the same for RQMC-WoS in this paper, obtaining very slightly different variance reductions from the regression-based ones in \cite{rqmcwos:tr}.

\subsection{Gasket example}\label{sec:gasket}

This example comes from \cite{wosoneweekend}.  The function $u$ gives the spatially varying temperature of a cylinder head gasket. The gasket has 50 holes in it corresponding to four cylinder bores and some other holes for coolant, oil and oil return. The boundary $\partial\Omega$ includes the exteriors of these holes as well as the outer boundary of the gasket itself. The function $u$ is harmonic, so $\Delta u = 0$ inside $\Omega$. The boundary value function is
\begin{equation}\label{eq:gasket}
b(\bsz) = \sum_{r\in\{
\mathrm{coolant},\ 
\mathrm{oil\ return},\ 
\mathrm{outer},\ 
\mathrm{oil},\ 
\mathrm{bore}\}}
T_r\,\bsone_r(\bsz)
\quad\text{for $\bsz\in\partial\Omega$}
\end{equation}
where they give example temperatures $T_r$ of 90, 110, 120, 130 and 160 degrees Celsius for the boundary elements in the order listed above.
As in~\cite{rqmcwos:tr}, we consider estimation of $u(\bsz_*)$ for $\bsz_*=(0.240999,0.3)$, which is just above the center of the third borehole and is nearly equidistant from several  curves with different temperatures.

We ran the MC-WoS, RQMC-WoS and Array-RQMC-WoS algorithms to estimate $u(\bsz_*)$ for the gasket with
$100$ independent replicates of each method, with $\varepsilon=10^{-3}$. 
The variance curves for these three methods are shown in Figure~\ref{fig:var-gasket}. The apparent convergence rate for RQMC-WoS is $O(n^{-1.09})$ and for Array-RQMC-WoS it is $O(n^{-1.62})$.  There is not much difference between methods using Sobol' points or lattices and so we have pooled their data to estimate rates. The RQMC-WoS variances follow their reference curve closely for both large and small $n$. The Array-RQMC-WoS variances for small $n$ lie below the reference line fit to $n\ge 128$. It appears that their slope has improved as $n$ has increased.  

Variance reduction factors for this problem are given in Table~\ref{tab:mserf-summaries} for sample sizes $n=2^{12}$ and $2^{17}$, respectively. The reductions are roughly 24:34-fold at $n=2^{12}$ and 200:318-fold at $n=2^{17}$ with the higher values being for Array-Lattice-WoS and lower values for Array-Sobol-WoS. Accuracy for our other examples is also reported in those tables. 


\begin{figure}
    \centering    \includegraphics[width=0.8\linewidth]{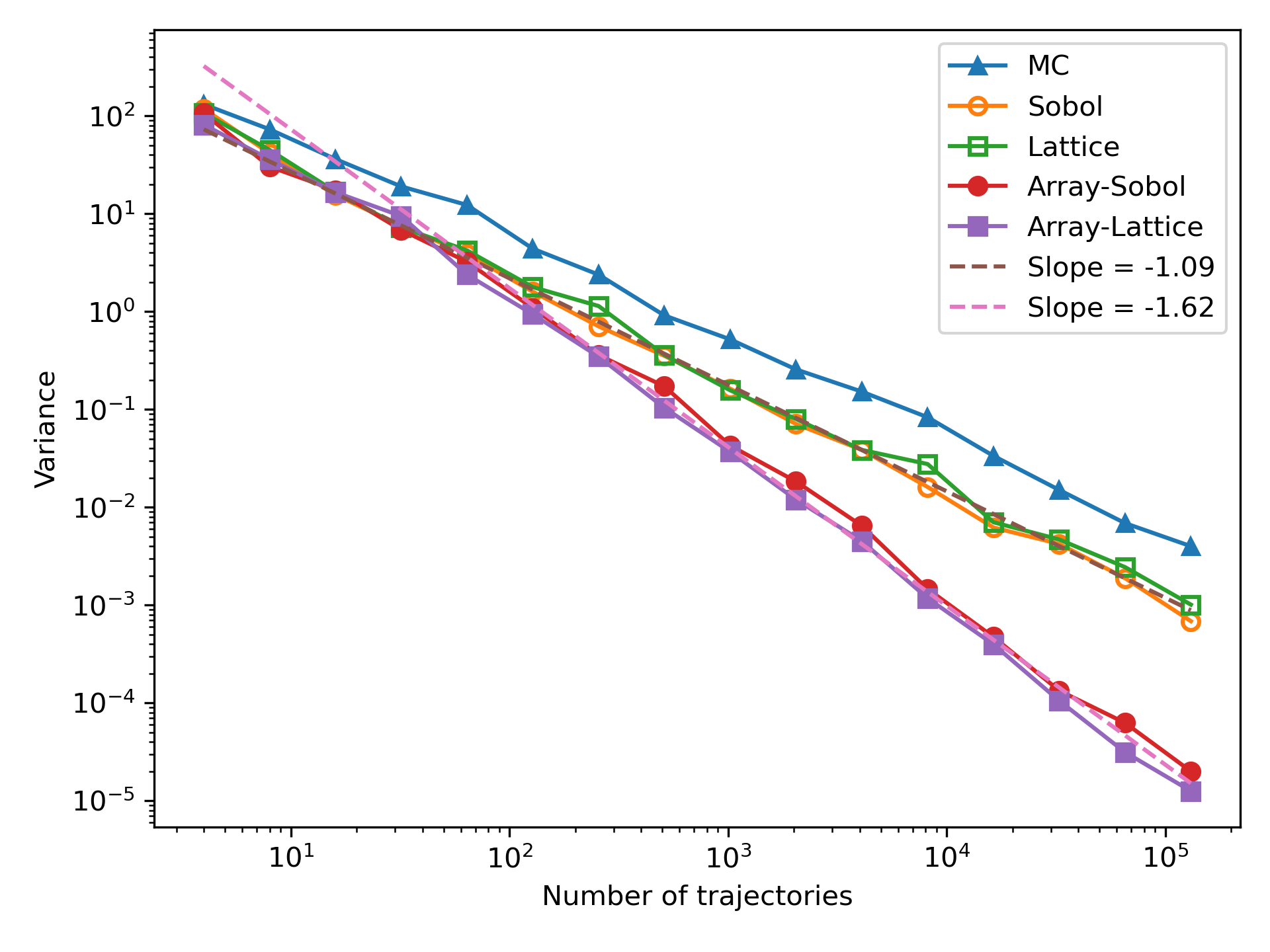}
    \caption{Variance of the Array-RQMC-WoS, RQMC-WoS and MC-WoS estimators, for the gasket example at $\bsz_0=\bsz_{*}$. The RQMC points used are scrambled Sobol' points and shifted lattices. Reference lines are as described in Section~\ref{sec:numerical}.}
    \label{fig:var-gasket}
\end{figure}

\subsection{Unit disk example}\label{sec:unitdisk}
This example comes from \cite{masc:hwan:2003} and \cite{masc:kara:hwan:2004} who describe it as computing an electrostatic potential.
The domain is $\Omega = B_2(\bszero,1)$.
The BVP has $\Delta u=0$ in $\Omega$ and
\begin{equation}\label{eq:unit-disk}
    u(\bsz) = \frac{1}{2}\ln[(z_1-2)^2+z_2^2],\quad\text{for  $\bsz \in \partial\Omega$.}
\end{equation}
The formula in~\eqref{eq:unit-disk} is also the solution inside $\Omega$, so this example allows one to compute an MSE, not just a variance.

The MSE curves are shown in Figure~\ref{fig:mse-disk}, for $\bsz_0=(0,0.5)$ and $\varepsilon=10^{-4}$. The RQMC-WoS and Array-RQMC-WoS estimators have MSEs close to their reference curves of $O(n^{-1.14})$ and $O(n^{-1.78})$, respectively. Once again the Array-RQMC-WoS MSEs follow a better slope for large $n$ than for small $n$.

\begin{figure}
    \centering    \includegraphics[width=0.8\linewidth]{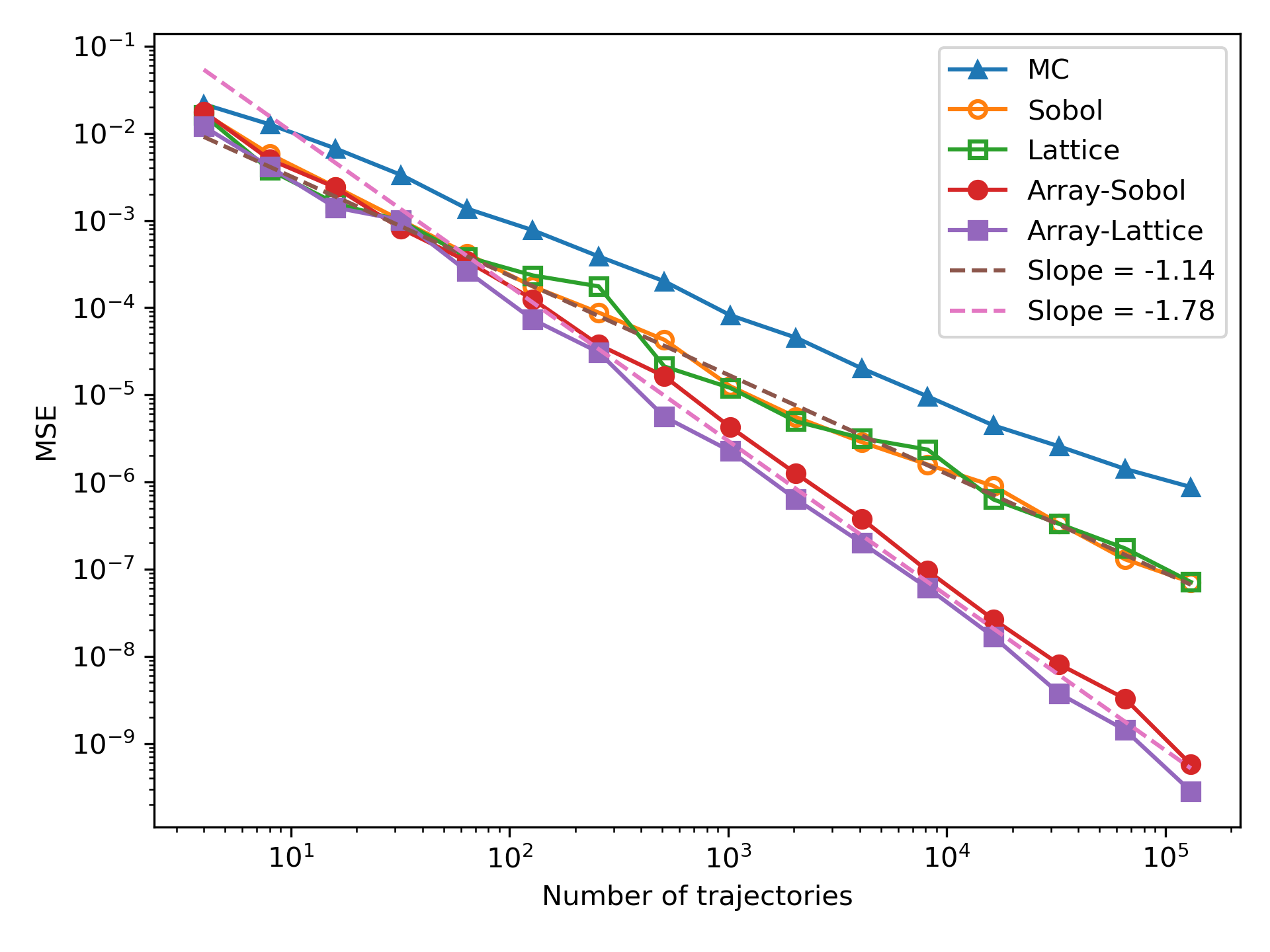}
    \caption{MSE of the Array-RQMC-WoS, RQMC-WoS and MC-WoS estimators for the unit disk example. The RQMC curves have Sobol' and lattice versions.}
    \label{fig:mse-disk}
\end{figure}

\subsection{Pac-Man example}
This example comes from \cite{dela:rome:1990} who do not give a physical motivation.
The domain, in polar coordinates, is
$$\Omega=\{(r,\theta)\mid 0\le r\le 1, -3\pi/2\le\theta\le 0\}.$$
Equivalently, $\Omega$ is the unit disk with the first quadrant removed, which resembles a rotation of the 1980s video game character Pac-Man. This domain has a re-entrant corner at the origin.  The function has a source term 
\begin{align}\label{eq:pacmansource}
\Delta u(r,\theta)=g(r,\theta)=-(2-r^2)e^{-r^2/2}\quad\text{for $(r,\theta) \in \Omega$.}
\end{align} 
The boundary conditions are
$b(r,0)=e^{-r^2/2}$,
$b(r,-3\pi/2)=-r^{1/3}+e^{-r^2/2}$ and
$b(1,\theta)= \sin(\theta/3)+e^{-1/2}$. The solution is known to be
$u(r,\theta)=r^{1/3}\sin(\theta/3)+e^{-r^2/2}$ in all of $\Omega$.


Figures \ref{fig:pacman-variance} and \ref{fig:pacman-mse} show the variance and MSE respectively. We see that Array-RQMC-WoS with the Korobov lattices has some poor behavior at sample sizes $2^6$, $2^{11}$ and $2^{15}$. We did not see that in other examples. We thought it might be due to the nonzero source term. To test this, we also used the same QMCPy lattices that we had used in \cite{rqmcwos:tr}. As documented there, those lattices, due to F.\ Kuo, are designed for $2^{10}\le n\le 2^{20}$ but also worked well on some smaller sample sizes in our examples. Figures~\ref{fig:pacman-variance} and~\ref{fig:pacman-mse} show no anomalies with that lattice. Non-smooth progress for lattices is not often remarked upon but was seen, for example, in \cite{lecu:mung:2012}. In most of our examples, the Korobov lattices gave better results but in this one the Kuo lattices were better.

The MSE for Array-RQMC-WoS starts to go above the reference curve at the largest sample sizes but the variance does not. That implies that the bias in Array-RQMC-WoS is no longer a negligible part of the error. In this example, the standard deviation goes below $10^{-4}$ when estimating the solution at $\bsz_0=(r_0 \cos(\theta_0),r_0 \sin(\theta_0))$ with $(r_0,\theta_0)=(0.1244,-0.7906)$, with $\varepsilon=10^{-4}$. This is one of the examples where \cite{masc:hwan:2003} found the bias to be~$O(\varepsilon)$.

\begin{figure}
    \centering

    \begin{subfigure}{0.75\linewidth}
        \centering
        \includegraphics[width=\linewidth]{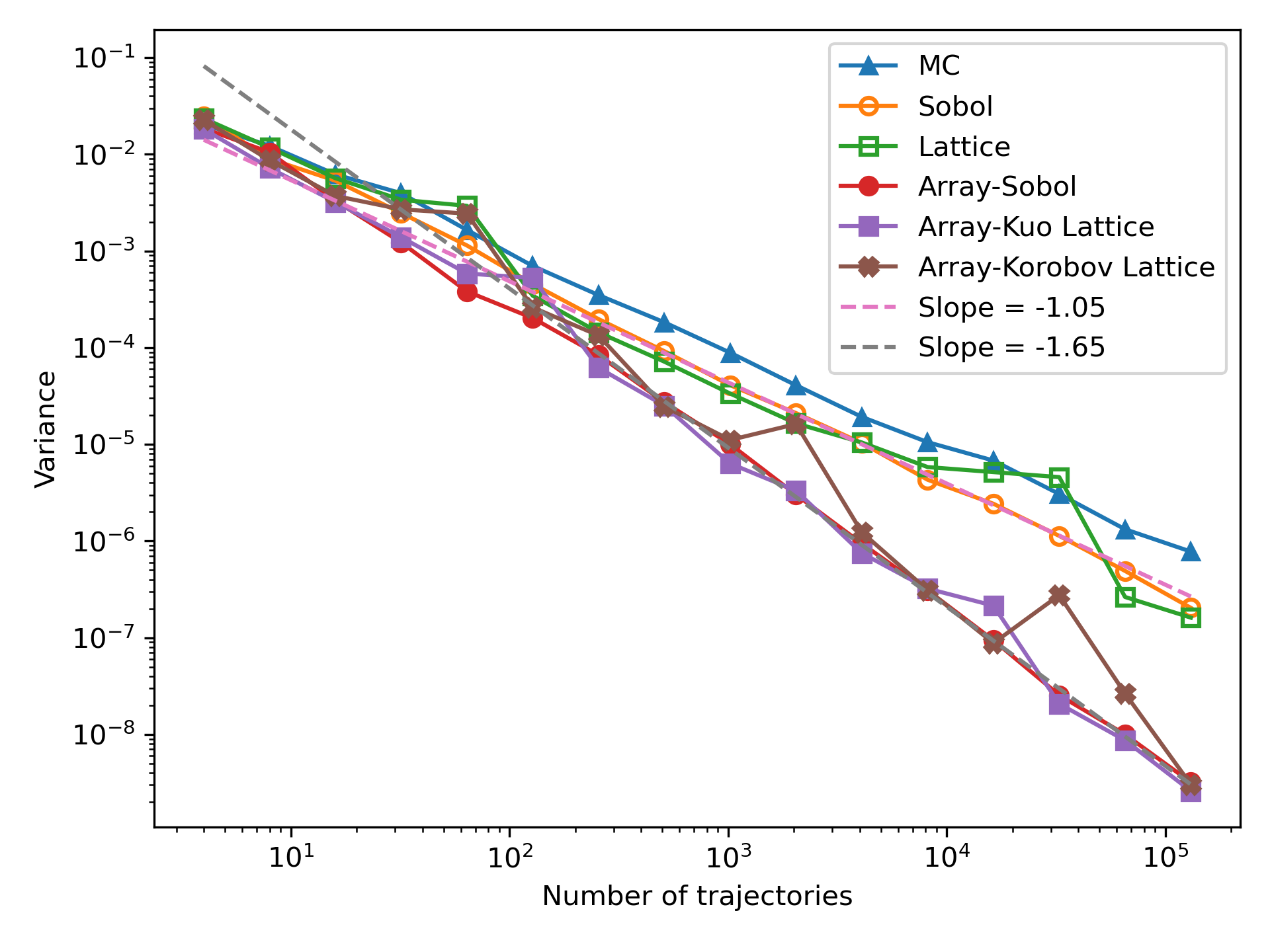}
        \caption{Variance for the Pac-Man example.}
        \label{fig:pacman-variance}
    \end{subfigure}

    \vspace{0.75em}

    \begin{subfigure}{0.75\linewidth}
        \centering
        \includegraphics[width=\linewidth]{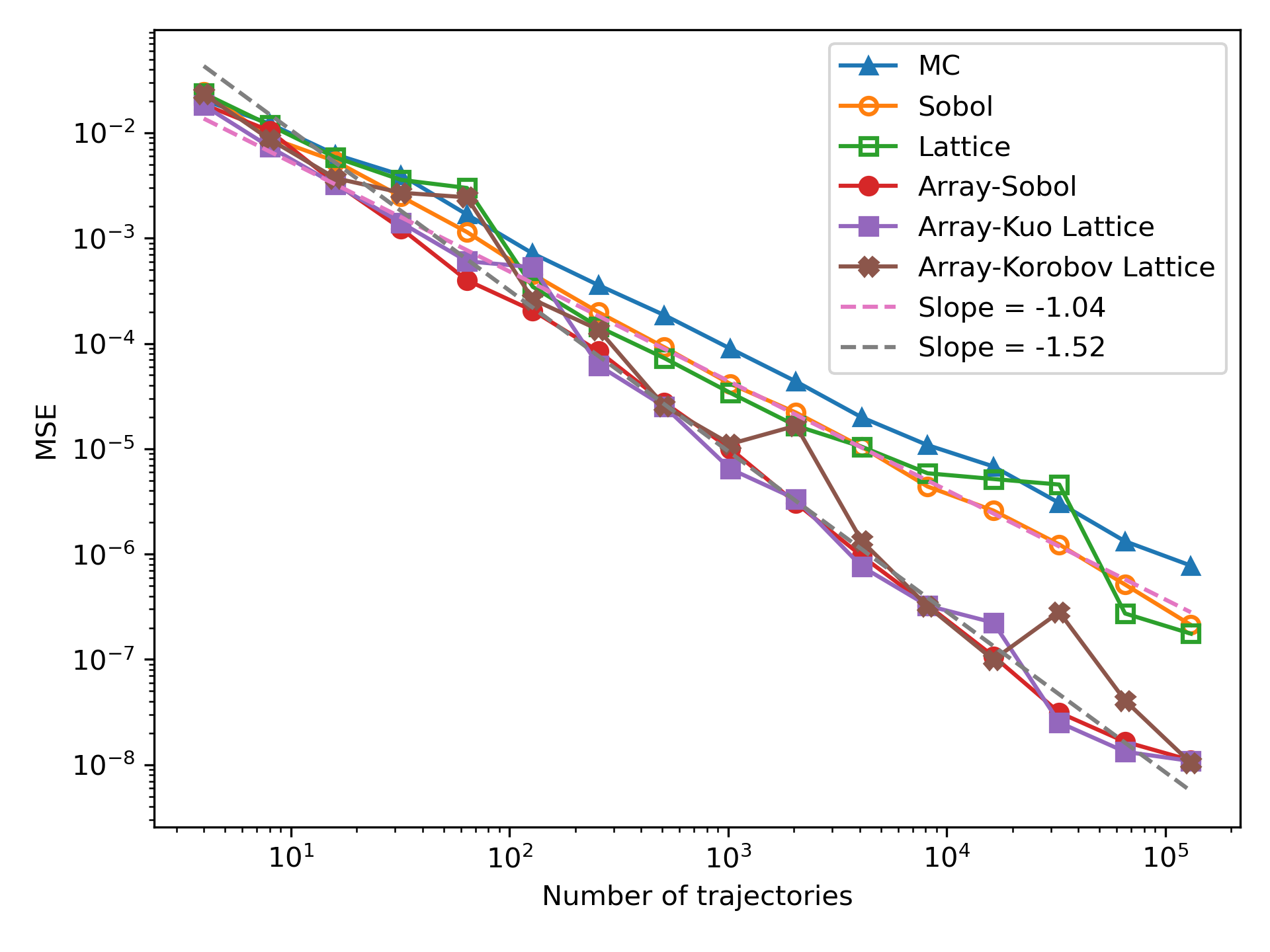}
        \caption{MSE for the Pac-Man example.}
        \label{fig:pacman-mse}
    \end{subfigure}

    \caption{Variance and MSE of the Array-RQMC-WoS, RQMC-WoS, and MC-WoS estimators for the Pac-Man example. The dashed reference curves show pooled fitted reference slopes. }
    \label{fig:var-mse-pacman}
\end{figure}

\subsection{Dumbbell example}\label{sec:dumbbell}

We considered another example with a nonzero source term taken from 
\cite{lund:rama:2019} and \cite{magn:pogg:2022}. They investigate, respectively, the number of critical points of $u$ and their distance to the boundary of their domain $\Omega$.

Their boundary value problem is
\begin{equation}\label{eq:dumbbell}
\Delta u(\bsz)=-2,\quad \bsz\in\Omega,
\qquad
u(\bsz)=0,\quad \bsz\in\partial\Omega .
\end{equation}
Here
$$
\Omega= B_2((-L,0),R)\cup([-L,L]\times[-w,w])\cup B_2((L,0),R),
$$
so that $\Omega$ consists of two disks of radius $R<L$ connected by a thin rectangular bridge of length $2L$ and width $2w$. We evaluate the solution at $\bsz_0=(L-R,0)$.
As discussed in \cite{magn:pogg:2022}, this problem can be interpreted physically as describing the velocity profile of a fluid flowing through a pipe whose cross-section is $\Omega$.

For this particular problem, the constant source term simplifies the WoS estimator. Rather than sampling an interior point inside a disk at each step, one only needs to add the contribution associated with the squared radius of the ball at each step, as described in \cite{rqmcwos:tr}. Thus,

$$\hat{u}(\bsz_0)=\frac{1}{2n}\sum_{i=1}^n\sum_{k=1}^{\tau_{i}}
\dist(\bsz_{i,k-1},\partial\Omega)^2
$$
where $\tau_i = \tau_{i,\varepsilon}=\min\{k\in\natu_0\mid \dist(\bsz_{ik},\partial\Omega)<\varepsilon\}$.
Results using RQMC-WoS, Array-RQMC-WoS and MC-WoS are shown in Figure~\ref{fig:dumbbell-array}.

\begin{figure}
    \centering
    \includegraphics[width=0.8\linewidth]{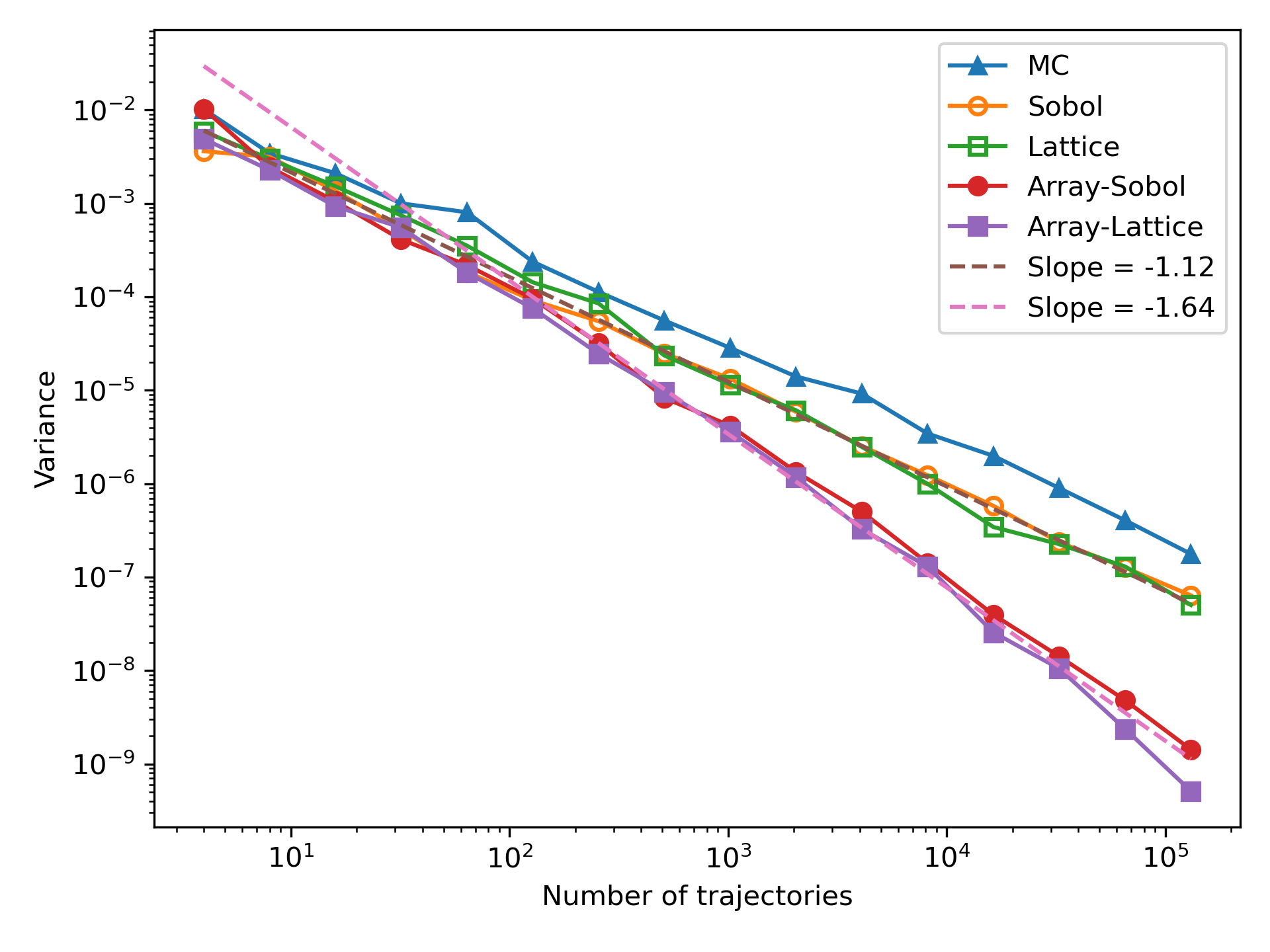}
    \caption{Variance for the dumbbell example, using Array-RQMC-WoS, RQMC-WoS and MC-WoS. In this example, $L=1.5$, $R=1.0$, $w=0.4$ and $\varepsilon=10^{-4}$.}
    \label{fig:dumbbell-array}
\end{figure}

\subsection{Unit ball example}\label{sec:unitball}
This example from \cite{masc:hwan:2003} has a three dimensional domain, $\Omega=B_3(\bszero,1)$. Therefore it takes $2$ uniform variables to generate $\bsz_k$ from $\bsz_{k-1}$. In Array-RQMC, we advance the walkers using the last two columns of a three dimensional RQMC point set whose first column we align with the Hilbert projections of $\bsz_{k-1}$.

This example has $\Delta u(\bsz)=0$ for $\bsz\in \Omega$. The boundary condition is
$$u(\bsz)=[(z_1-2)^2+z_2^2+z_3^2]^{-1/2}$$ for $\Vert\bsz\Vert=1$
which is also the solution inside $\Omega$.

Figure \ref{fig:mse-sphere} shows the MSE versus $n$ for estimating $u(\bsz_0)$ at $\bsz_0=(0.2, 0.3, -0.1)$ with $\varepsilon=10^{-4}$. The RQMC-WoS and Array-RQMC-WoS reference curves are $O(n^{-1.15})$ and $O(n^{-1.42})$, respectively.  The MSEs do not curve upwards near the end. This is consistent with the bias being negligible for $n\le2^{17}$, unlike the Pac-Man case where we saw evidence of bias.  The MSE for Array-RQMC-WoS lies below the reference line for small $n$. This is consistent with a slowly improving convergence rate for Array-RQMC-WoS as $n$ increases.  

\begin{figure}
    \centering
    \includegraphics[width=0.8\linewidth]{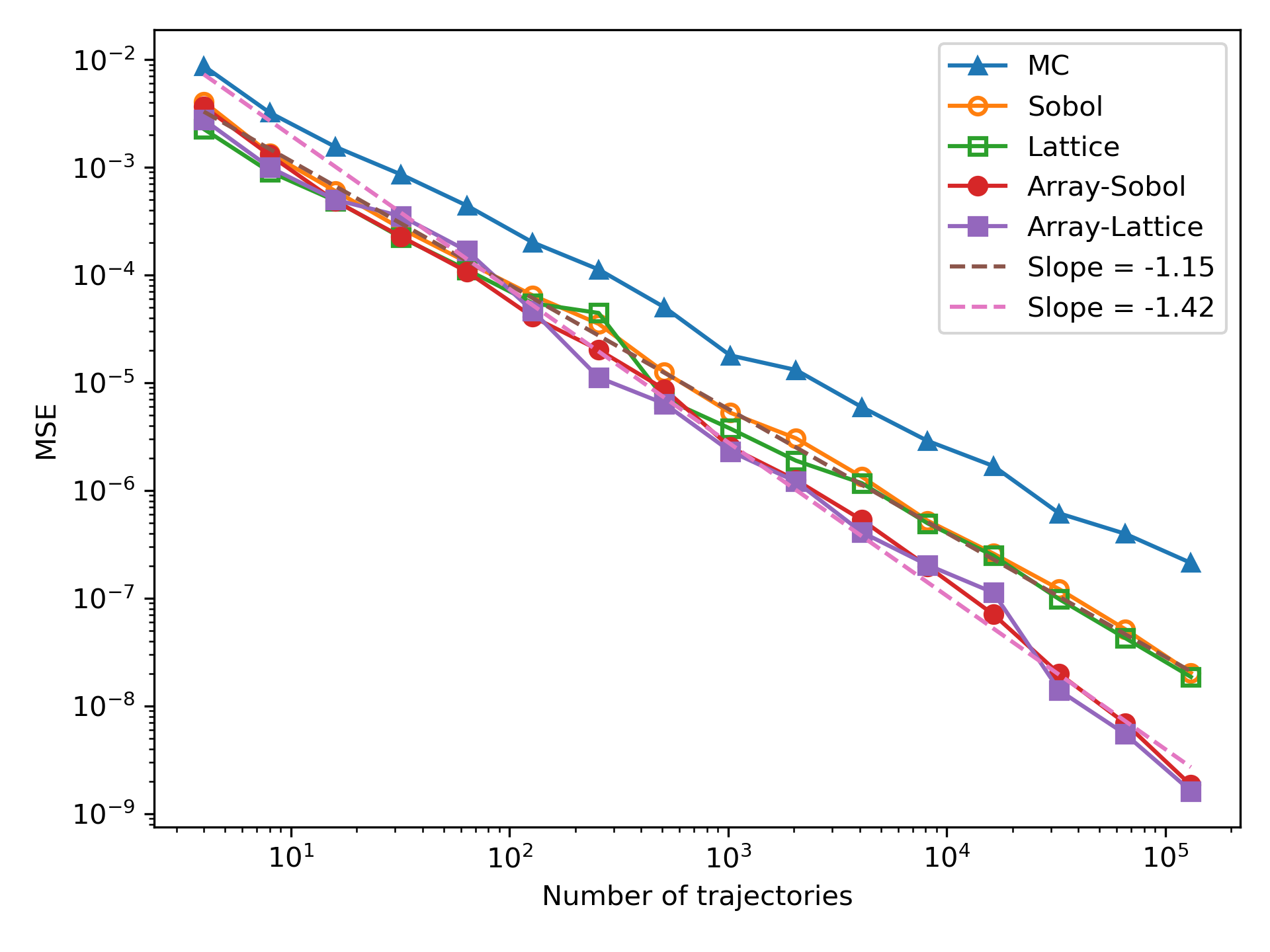}
    \caption{
    MSE of MC-WoS, RQMC-WoS and Array-RQMC-WoS estimators for the unit ball example.}
    \label{fig:mse-sphere}
\end{figure}

\subsection{Summary of examples}\label{sec:summary}

We show MSE reduction factors for the unit disk, unit ball and Pac-Man examples as well as variance reduction factors for the gasket and dumbbell examples, when using Sobol and Lattice-WoS, Sobol and Lattice Array-WoS, compared to MC-WoS, for sample sizes $n=2^{12}$ and $n=2^{17}$ in Table \ref{tab:mserf-summaries}. For the Pac-Man example, the Array-Lattice-WoS entries in Table \ref{tab:mserf-summaries} use the Kuo lattice rule, while the other Array-Lattice-WoS entries use the Korobov lattice rules described in Section \ref{sec:sampling}. We see that Array-Lattice-WoS is usually better than Array-Sobol-WoS although that improvement is by a modest factor.

There is very little theory to explain why Array-RQMC-WoS works so well, much less why it works differently in different problems.  We can mention some qualitative features of the problems. The best improvements we see are for the unit disk problem. It has a simple domain, a smooth boundary value function and no source term. The gasket problem has a disconnected boundary that adds discontinuities to the integrand and it is reasonable to expect that to be adverse to the MSE ratio. The unit ball example is also very smooth with a simple geometry but every step uses two uniform inputs instead of one and the mapping $\psi_{0,3}$ from the square to the sphere is not smooth. The Pac-Man problem has a source term which then consumes two more uniform inputs per step and it has a re-entrant corner. 
\begin{table}[t]
\centering
\small
\setlength{\tabcolsep}{4pt}
\begin{tabular}{lcccc}
\toprule
$n=2^{12}$ 
& Sobol-WoS
& Lattice-WoS
& Array-Sobol-WoS
& Array-Lattice-WoS \\
\midrule
Unit disk
& $\phz7.0$
& $\phz6.3$
& $53.2$
& $100.7$\\
Gasket
& $\phz3.9$
& $\phz3.9$
& $23.5$
& $\phz33.7$\\ 
Unit ball
& $\phz4.4$
& $\phz5.1$
& $11.2$
& $\phz14.5$\\
Pac-Man
& $\phz1.9$
& $\phz1.9$
& $20.7$
& $\phz26.1$\\
Dumbbell
&$\phz3.6$
&$\phz3.7$
&$18.3$
&$\phz28.2$\\
\bottomrule\\[1ex]
\toprule
$n=2^{17}$
& Sobol-WoS
& Lattice-WoS
& Array-Sobol-WoS
& Array-Lattice-WoS \\
\midrule
Unit disk
& $12.3$
& $12.3$
& $1519.1$
& $3086.8$\\
Gasket
& $\phz5.8$
& $\phz4.0$
& $\phz200.0$
& $\phz 318.2$\\
Unit ball
& $10.5$
& $11.4$
& $\phz115.2$
& $\phz 131.0$\\
Pac-Man
& $\phz3.7$
& $\phz4.4$
& $\phz\phz70.7$
& $\phz\phz71.6$ \\ 
Dumbbell
&$\phz2.8$
&$\phz3.5$
&$\phz124.0$
&$\phz 345.3$\\
\bottomrule
\end{tabular}
\caption{Estimated MSE reduction factors at $n=2^{12}$ and $n=2^{17}$, versus MC-WoS, over
$100$ replicates. 
For the gasket and dumbbell examples, we report the variance reduction factor instead.
}
\label{tab:mserf-summaries}
\end{table}

\section{Algorithmic variants}\label{sec:variants}
In our examples, Array-RQMC-WoS regularly outperformed RQMC-WoS by a large factor and RQMC-WoS in turn consistently outperformed MC-WoS, but by a much smaller factor. We explored some variations on Array-RQMC-WoS to see if any further large improvement factors could be obtained. Section~\ref{sec:fibolattice} considers replacing Korobov lattices by Fibonacci lattices. Section~\ref{sec:converged} considers alternatives to the move-to-the-end strategy for handling converged points. Section~\ref{sec:large-qmc} considers using a single high-dimensional RQMC point set, instead of a fresh lower-dimensional RQMC point set at each step. Section~\ref{sec:fixedk} considers running Array-RQMC-WoS for a fixed number of steps instead of using a fixed stopping distance $\varepsilon$.

\subsection{Fibonacci lattices}\label{sec:fibolattice}

In our examples, the Korobov lattices are constructed using sample sizes that are powers of two. That makes it easy to compare them to scrambled Sobol' points.  However, when $d=2$ and there is no source function, the Fibonacci lattices are available and those are known to be especially good choices. For one such result, see \cite{bily:teml:vald:yu:2012} who consider $L_2$ discrepancy. 

The Fibonacci lattice has
\begin{align}\label{eq:fibo}
\bsx_i = \biggl( \frac{i-1}n, \biggl\{\frac{(i-1)F_{r-1}}{n}\biggr\}\biggr)
\end{align}
for $i=1,\dots,n=F_r$ where $F_r$ is the $r$th Fibonacci number.  As usual for lattices, $\{v\}=v-\lfloor v\rfloor$ denotes the fractional part of $v\in\real$ and context will distinguish that usage from the set containing $v$.

Step $k$ of Array-RQMC-WoS only uses the first $m_k\le n$ points of the sequence so we might expect Fibonacci lattices to have no particular advantage especially in the later steps.  On the other hand, the early steps with $n$ close to $m_k$ are expected to be the most important ones, so perhaps a Fibonacci lattice will be better.

We made one comparison of Fibonacci lattices with Korobov lattices. This was for the unit disk integrand starting at $\bsz_0 =(0,0.5)$ using $100$ replicates.  The results in Figure~\ref{fig:fibo} show essentially equivalent performance for Korobov lattices and Fibonacci lattices and we did not think that the issue needed more investigation, so we retain the sample sizes $n=2^m$ as they allow ready comparison with Sobol' points.

\begin{figure}
    \centering
   \includegraphics[width=0.8\linewidth]{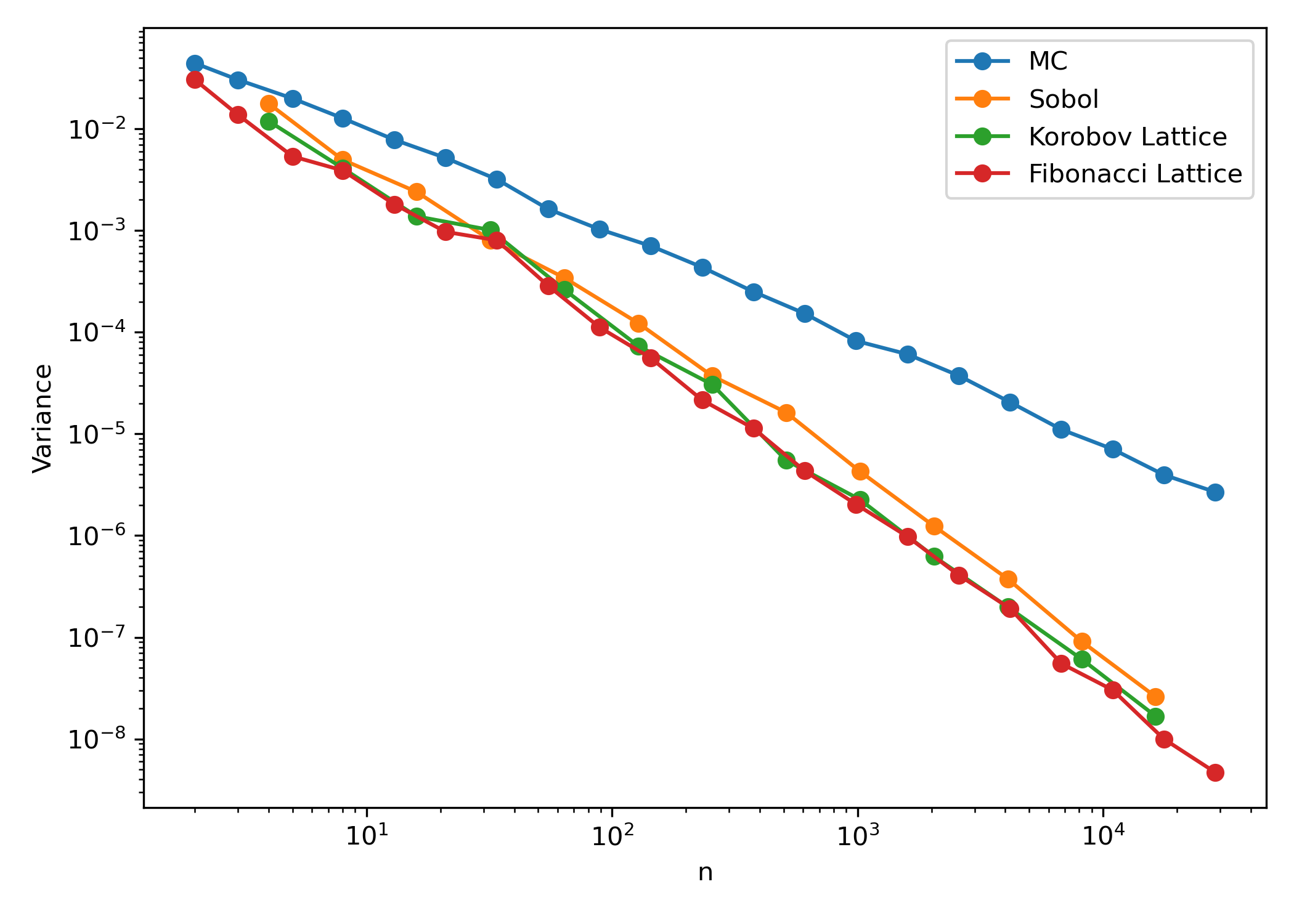}
    \caption{Variance curves for the Array-RQMC-WoS estimator using scrambled Sobol' points,  Korobov lattices and Fibonacci lattices, for the unit disk example at $\bsz_0=(0,0.5)$.}
    \label{fig:fibo}
\end{figure}

\subsection{Converged points in Array-RQMC-WoS}\label{sec:converged}

At step $k$ of the algorithm, we must sample $n$ values for $\bsz_{1k},\dots,\bsz_{nk}$ starting at points $\bsz_{i,k-1}$.  A complication is that some of the $\bsz_{i,k-1}$ have already entered $\partial\Omega_\varepsilon$.  If $m_k$ points have not converged, then we only need $m_k$ of the $n$ RQMC points to advance the trajectories.  The strategy in \cite{l2018sorting} is to move the Hilbert key of the converged points to $+\infty$. Then for the $m_k$ remaining active walks, they use the first $m_k$ RQMC inputs of a fresh $n$-RQMC point set to advance to the next points.

The first $m_k$ of $n$ points will ordinarily not have as small a discrepancy as one could get with a customized $m_k$-point rule. For example $n$ might be a power of $2$ or a prime number where we get quite good discrepancy while $m_k$ will not generally be special. We might expect the discrepancy to be higher by $O(\log(m_k))$ because that is the factor between rates for low discrepancy sequences and low discrepancy point sets. Therefore we explored ways to construct $m_k$ point rules ``on the fly''.  When the squared error is $o(1/n^2)$, then not using special sample sizes can even worsen the rate of convergence \citep{firstsobol}, but we have not seen Array-RQMC-WoS attain errors that small.

First, we considered using all $n$ points to generate WoS movement directions while keeping the converged points where they were. That is equivalent to moving them a distance of zero in their assigned directions.  Those points also do not get new source function contributions.  In this strategy, the inactive walks retain their positions on the Hilbert curve with sequences of inactive walks interleaved between sequences of active walks. We tried the interleaving strategy for the gasket example of Section \ref{sec:gasket}  as well as the unit disk of Section~\ref{sec:unitdisk} and the unit ball of Section~\ref{sec:unitball}.  The results were very nearly equivalent to the original strategy, with no important advantage or disadvantage.

Another strategy we tried was to use Halton point sets of size $m_k$ \citep{halt:1960}. If we attach a first column of values $(i-1)/m_k$ to them, we get Hammersley points \citep{hamm:1960}.  For $d=2$ and no source term, we only need one base $2$ column of Halton points. For $d=3$ and no source term, we use bases $2$ and $3$. For $d=2$, we found this to be slightly worse than the interleaving and move-to-the-end methods. For the unit ball example, we saw nearly equivalent performance to the original method using the Hammersley strategy. 

As noted above, when $d=2$ and there is no source term, we can use randomly shifted Fibonacci lattices. Instead of using the same lattice size at every step, we can find the smallest Fibonacci number $F_r \geq m_k$.  Then, we use the first $m_k$ points of the lattice given by \eqref{eq:fibo}. For the gasket example, using Fibonacci lattices together with the move-to-the-end strategy gives accuracy comparable to the corresponding move-to-the-end method with Korobov lattices, as shown in Figure~\ref{fig:variants-gasket-cst}.

Finally, for $d=2$, we tested stratified point sets of size $m_k$ at each step $k$. Specifically, we used
$$x_{ik}=\frac{(a_k (i-1) \bmod m_k)+\Delta_k}{m_k}, \qquad i=1,\dots,m_k $$ where $a_k$ is coprime with $m_k$ and $\Delta_k\sim \runif(0,1)$ is a single random shift applied to all points in the RQMC point set at each step.  This construction places exactly one point in each interval $[j/m_k,(j+1)/m_k)$, $j=0,\dots,m_k-1$ and applies a shared shift within each bin. In our experiments, we chose $a_k$ by starting from $a_{k,0}=\max\left(\big\lfloor m_k/\varphi \big\rceil \bmod m_k, 2\right)$ where $\varphi$ is the golden ratio, and incrementing $a_k$ (wrapping back to $2$ if $a_k\geq m_k$) until $\gcd(a_k,m_k)=1$. Here $\lfloor x\rceil$ denotes the nearest integer to $x$. The use of the golden ratio scaling provides a simple rule of thumb that avoids very small multipliers, which can yield point sets concentrated on only a few diagonal strips.

\begin{figure}
    \centering
    \includegraphics[width=0.8\linewidth]{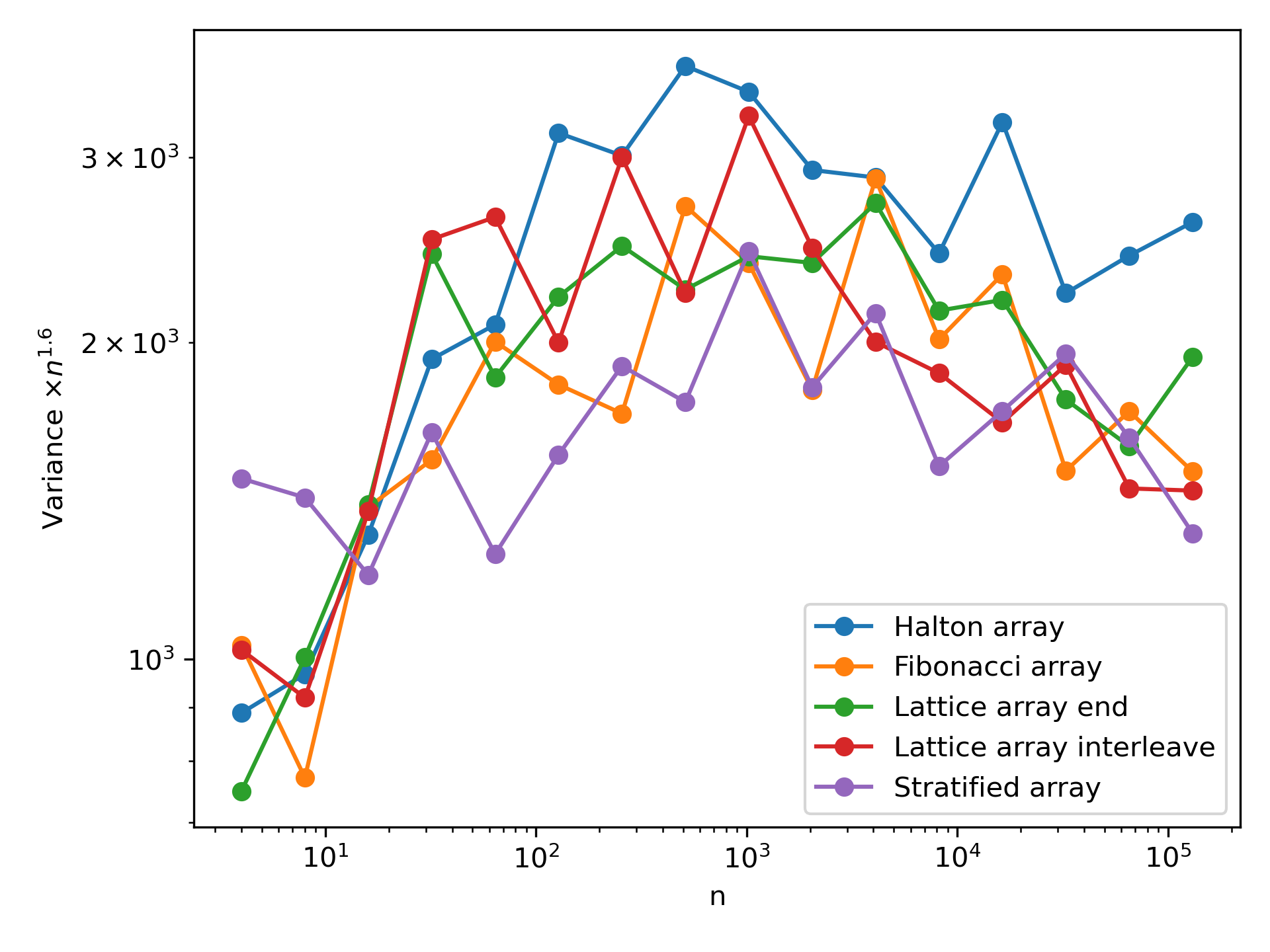}
    \caption{Variance curves for some variants of Array-RQMC-WoS discussed in Section \ref{sec:converged} for the gasket example, multiplied by $n^{1.6}$. 
}
    \label{fig:variants-gasket-cst}
\end{figure}

All of the variants we considered had variances just slightly better than $O(n^{-1.6})$, for the gasket example.  Figure~\ref{fig:variants-gasket-cst} shows a graphical comparison.  The variances are within a factor of $3$  of each other for $4\le n\le 2^{17}$. One of them might have a systematic advantage but it is likely to be negligible compared to the differences between MC-WoS, RQMC-WoS and Array-RQMC-WoS.

\subsection{Randomizing one large QMC point set}\label{sec:large-qmc}

Array-RQMC uses repeated randomizations of the same QMC point set in $[0,1)^{n\times s}$ at each step $k=1,\dots,K$. When we combine two randomizations of an $s$-dimensional QMC point set, the resulting $2s$-dimensional point set does not have low discrepancy. That discrepancy does not converge to $0$ as $n\to\infty$ for either repeated scrambles of a digital net or repeated random shifts of a lattice. However, the stepwise argument in \cite{gerber2015sequential} does not require that joint discrepancy to converge to zero.


We replaced the $K$ independent randomizations of an $s$-dimensional point set with an $sK$-dimensional RQMC point set using columns $(k-1)s+1$ through $ks$ to take step $k$ to test whether it would bring an improvement. The numerical results were not materially different from Array-RQMC-WoS.  A potential disadvantage of using an $sK$-dimensional RQMC point set is that very high dimensional point sets can have worse equidistribution.  Our examples did not use a large $K$ and $s$ was always small.

\subsection{Fixed step algorithms}\label{sec:fixedk}

The number $\tau$ of steps to convergence is known to scale as $\log(1/\varepsilon)$ \citep{bind:brav:2012}. If one simply runs the algorithm to a large value of $k$, the great majority of points $\bsz_{ik}$ will be close to $\partial\Omega$ although some may be far from the boundary. A fixed $k$ algorithm was considered in \cite{wu:etal:2025}. \cite{rqmcwos:tr} studied a fixed $k$ algorithm because, for planar domains, $\bsx \mapsto \bsz_k(\bsx)$ is periodic and Lipschitz, and fixing $k$ removes the discontinuity due to the stopping rule.  
In the end, it did not give lattice methods an advantage over digital nets. We investigated a $k=20$ step version of Array-RQMC-WoS for the unit disk problem and the gasket problem.  We did not see a meaningful difference between fixed $k$ and fixed $\varepsilon$ stopping rules, in terms of variance.

\section{Discrepancy and mean dimension explanations}\label{sec:explanations}

We would like to understand why Array-RQMC-WoS is so much more effective than RQMC-WoS. For $\Omega=B_2(\bszero,1)$, we know the proper distribution for $\bsz_\tau$ and so we can compute the discrepancy of sample points from MC-WoS, RQMC-WoS and Array-RQMC-WoS with respect to that distribution.

Next, we define a notion of mean dimension suitable for Array-RQMC and see that Array-RQMC-WoS is better than RQMC-WoS at exploiting low mean dimension in the gasket example. Intuitively, the methods correctly integrate part of the integrand leaving an error from the remainder.  If the mean dimension of the remainder is high that is a sign that the algorithm has done a better job at integrating low dimensional structure.

\subsection{Discrepancy under MC, RQMC and Array-RQMC}\label{sec:discrepancy}

Array-RQMC-WoS brings together two innovations: the use of RQMC and the use of array sampling. We would like to measure the contributions of those two innovations separately. For the case where $\Omega$ is the unit disk $B_2(\bszero,1)$, we know that
$u(\bsz) = \int_0^1 p_{\bsz}(x) b(\theta(x))\rd x$
where $\theta(x) = (\cos(2\pi x),\sin(2\pi x))^\tran$ and
\begin{align}\label{eq:poisson}
p_{\bsz}(x)=\frac{1-\Vert\bsz\Vert^2}{\Vert\bsz - \theta(x)\Vert^2} \end{align}
is the Poisson kernel. 
A good sampling algorithm will terminate at a set of $n$ points $\bsz_i$ whose angle in polar coordinates, divided by $2\pi$, has an empirical distribution $\hat p_{\bsz}$ close to $p_{\bsz}$.

For $n = 2^m$ with $2\le m\le 14$ and $\bsz=(t,0)^\tran$ for $t\in\{1/3,1/2,3/4,9/10\}$ we computed the Kolmogorov-Smirnov (KS) distance between $p_{\bsz}$ and $\hat p_{\bsz}$ $100$ times. The KS distance in one dimension is the same as the star discrepancy, so the best possible value from $n$ points is $1/(2n)$.  For a plain Monte Carlo sample from $p_{\bsz}$, the KS distance has expected value $\sqrt{\pi/2}\log(2)/\sqrt{n}$ \citep{mars:etal:2003}. Figure \ref{fig:ks-distance} shows the results with reference curves for plain MC sampling and for the best KS distance. Perhaps surprisingly, the discrepancy for MC-WoS is usually very close to what we would have gotten from IID sampling from $p_{\bsz}$. There are some exceptions for the largest $t$ and largest $n$ where MC-WoS underperforms that benchmark.  RQMC-WoS reduces the discrepancy compared to MC-WoS and Array-RQMC-WoS reduces it further. None of the methods attain the optimal rate. We also see that for larger values of $t$, KS distances are higher.  The discrepancy advantage of Array-RQMC versus RQMC is not present for the smallest sample sizes although both outperform MC there.

\begin{figure}[t]
    \centering
    \includegraphics[width=1\linewidth]{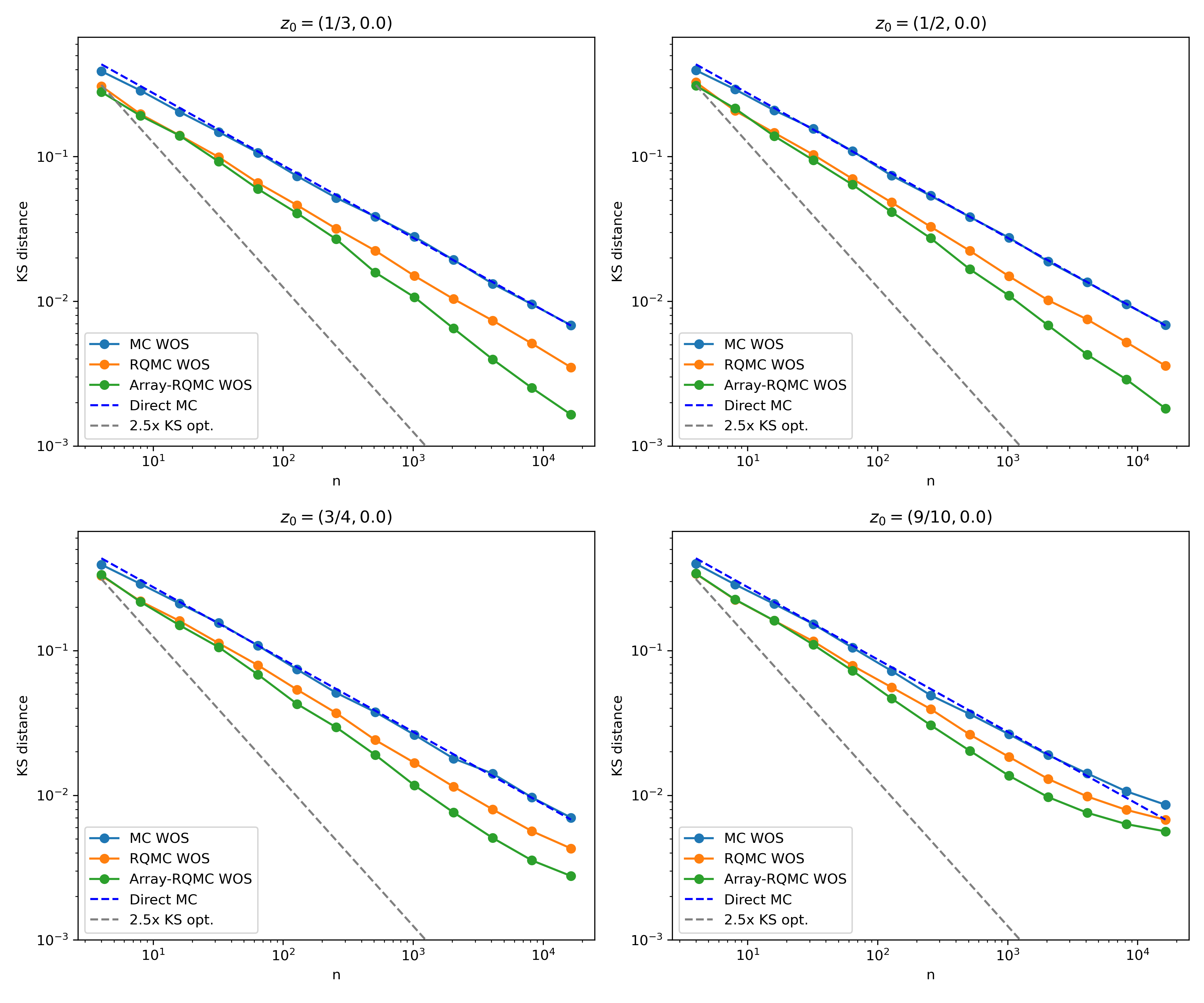}
    \caption{Kolmogorov-Smirnov distance between $p_{\bsz}$ and $\hat p_{\bsz}$ averaged over $100$ replicates, for $4$ different starting points in the unit disk. The gray dashed line is proportional to the optimal values $1/(2n)$. The RQMC points we use here are scrambled Sobol' points.}
    \label{fig:ks-distance}
\end{figure}

\subsection{ANOVA,  Sobol' indices and mean dimension}

RQMC is known to benefit greatly when integrands have a low effective dimension \citep{cafmowen}.  If the integrand is dominated by the main effects and low-order interactions in its ANOVA decomposition, that combines well with low-dimensional coordinate projections of RQMC points having much lower discrepancy than Monte Carlo points have in those dimensions.  There is also an effect that often makes low order ANOVA components smoother than the original integrand \citep{grie:kuo:sloa:2010}.

Notions of effective dimension are conveniently defined through the ANOVA decomposition of
$L_2[0,1]^d$. Let $f$ be a real-valued function on $[0,1]^d$, with $\mu= \int_{[0,1]^d}f(\bsx)\rd\bsx$ and $\sigma^2= \int_{[0,1]^d}(f(\bsx)-\mu)^2\rd\bsx<\infty$. We write
\begin{align}\label{eq:anova}
f(\bsx) = \sum_{u\subseteq1:d}f_u(\bsx)
\end{align}
where the effect $f_u$ depends on $\bsx$ only through $\bsx_u\in[0,1]^{|u|}$ and is defined recursively by
$$
f_u(\bsx) = \int_{[0,1]^{d-|u|}}\biggl(f(\bsx)-\sum_{v\subsetneq u}f_v(\bsx)\biggr)\rd\bsx_{-u}
= \int_{[0,1]^{d-|u|}} f(\bsx)\rd\bsx_{-u}
-\sum_{v\subsetneq u}f_v(\bsx).$$
The recursion starts with $f_\emptyset$,  the constant function equal to $\mu$ everywhere. The effect $f_u$ has variance component $\sigma^2_u=\var(f_u(\bsx))$. The ANOVA is very interpretable because $\sigma^2 = \sum_{u\subseteq1:d}\sigma^2_u$ shows how the importance of the components of $\bsx$ can be decomposed over subsets of variables.

The Sobol' indices of $f$ 
$$
\ult^2_u = \sum_{v\subseteq u}\sigma^2_v\quad\text{and}\quad\olt^2_u = \sum_{v\cap u\ne\emptyset}\sigma^2_v
$$
provide two different ways to quantify the importance of $\bsx_u$.  Normalized versions $\uls_u=\ult^2_u/\sigma^2$ and $\ols_u=\olt^2_u/\sigma^2$ describe the proportion of variance attributable to subsets. The mean dimension of $f$ in the superposition sense is defined as
$$
\nu(f) = \sum_{u\subseteq1:d}\sigma^2_u|u|
\Bigm/\sum_{u\subseteq1:d}\sigma^2_u = \sum_{j=1}^d \frac{\olt^2_j}{\sigma^2}
$$
with the final identity based on a result in \cite{meandim}. This identity
$$
\olt^2_j = \frac12\e\bigl( (f(\bsx)-f(\bsx_{-j}{:}z_j))^2\bigr)
$$
from \cite{jans:1999}
allows us to get an unbiased estimate of $\olt^2_j$ as an expectation over $d+1$ variables without ever constructing estimates of the ANOVA effects $f_u$. In the last identity, $\bsx_{-j}{:}z_j$ denotes the hybrid point obtained from $\bsx$ by replacing its $j$th coordinate $x_j$ with $z_j$, where $\bsz$ is an independent sample from the same distribution as $\bsx$. The mean dimension is very easy to estimate compared to the effective dimension of \cite{cafmowen} which is the smallest integer $s$ with $\sum_{|u|\le s}\sigma^2_u\ge 0.99\sigma^2$.

\subsection{Vector-wise Sobol' indices}

In Array-RQMC-WoS we cannot replace a value $x_{ik}$ by an independently sampled value $x_{ik}^*$, because that would break the equidistribution in column $k$.  What we can do is replace the whole vector $\vec{\bsx}_k=(x_{1k},x_{2k},\dots,x_{nk})^\tran$ by an independently sampled vector $\vec{\bsx}_{k}^*\in[0,1)^n$.  That will give us a measure of the importance of step $k$ to the outcome. 

The ANOVA can be defined for any $L_2$ function of independent quantities.  Those quantities do not have to be real numbers. One generalization is to partition a set of variables into groups and define an ANOVA at group level.  \cite{sobo:1993} mentions this possibility and \cite{salt:etal:2008} describe how it can lead to computational savings in high dimensional settings.  \cite{jans:ross:walt:daam:1994} describe an early application to agronomy.

We let $\cx\in[0,1)^{n\times sK}$ represent the entire matrix of values that drive the WoS computation. We write $\cx=\begin{pmatrix}\vec{\bsx}_1,\vec{\bsx}_2,\cdots,\vec{\bsx}_{sK}
\end{pmatrix}$. Then, we view
$\hat\mu$ as a function on an $sK$-fold Cartesian product of $[0,1)^n$. All of our algorithms sample $\vec{\bsx}_k$ independently of each other but RQMC versions include dependence within each $\vec{\bsx}_k$.


The ANOVA decomposition is defined for any square integrable function of $sK$ independent quantities.  Our Array-RQMC-WoS algorithm then provides a function
$F:V_n^{sK}\to\real$
where $V_n = [0,1)^n$ and $F(\cx)=\hat\mu$.
This function $F$ has an ANOVA decomposition with effects $F_u:V_n^{sK}\to\real$, variance components $\sigma^2_u(F)$ and Sobol' indices $\ult^2_u(F)$, $\olt^2_u(F)$ for $u\subseteq1{:}sK$.
Now, we define a vector-wise mean dimension as
$$
\nu(F)=\frac1{\sigma^2(F)}\sum_{k=1}^{sK}\olt^2_k(F)\quad\text{for}\quad
\olt^2_k(F) = \frac12\e\bigl[ (\hat\mu(\cx)-\hat\mu(\cx_{-k}{:}\cx_{k}^*))^2\bigr].
$$

It is instructive to see how the vector-wise mean dimension operates for MC and RQMC.
For MC
$$
F(\cx) = \frac1n\sum_{i=1}^nf(\bsx_i)
= \frac1n\sum_{i=1}^n\sum_{u\subseteq1{:}sK}f_u(\bsx_i)
$$
and then
\begin{align*}
\int_{V_n^{sK-|u|}}F(\cx)\rd \cx_{-u}
&=\int_{[0,1)^{n\times (sK-|u|)}}\frac1n\sum_{i=1}^nf(\bsx_i)
\prod_{i=1}^n\prod_{k\not\in u}\rd x_{ik}\\
&=\frac1n\sum_{i=1}^n\int_{[0,1)^{sK-|u|}}f(\bsx_i)
\prod_{k\not\in u}\rd x_{ik}
\end{align*}
from which $F_u(\cx)=(1/n)\sum_{i=1}^nf_u(\bsx_i)$.
It follows that
$\sigma^2_u(F) = \sigma^2_u(f)/n$, so
$\uls_u(F)=\uls_u(f)$ and $\ols_u(F)=\ols_u(f)$.
This implies that the mean dimensions coincide and so
$$\nu_{\mc}(F) = \nu_{\mc}(f).$$

For randomly shifted lattices, we can write $f(\bsx)=f_0(\bsx)+f_1(\bsx)$ where $f_0$ and $f_1$ have variances $\sigma^2_0$ and $\sigma^2_1$ respectively under MC sampling. Here $f_1$ is a sum of Fourier contributions over a dual lattice (minus the origin) and $f_0$ has all other Fourier contributions. Then $\var(\hat\mu_{\mc})=(\sigma^2_0+\sigma^2_1)/n$ while $\var(\hat\mu_{\rqmc})=\sigma^2_1$ (not divided by $n$). See \cite{lecu:lemi:2000}. Under randomly shifted lattice rules
$$
\nu_{\lattice}(F)=\nu_{\mc}(f_1).
$$
The function $f_1$ contains high-frequency components of $f$ and those are generally harder to integrate and we then expect typical problems to have $\nu_{\lattice}(F)=\nu_{\mc}(f_1)>\nu_{\mc}(f)=\nu_{\mc}(F)$.

The situation for scrambled nets is a bit more complicated. 
Then $\var(\hat\mu)=\sum_{|u|>0}\var(\hat\mu_u)$ where $\hat\mu_u = (1/n)\sum_{i=1}^nf_u(\bsx_i)$.
We can write $\var(\hat\mu_u)=\tilde\Gamma_u\sigma^2_u/n$. The terms $\tilde\Gamma_u$ depend on $n$ and $f$ but are subject to upper bounds $\Gamma_u$, known as gain constants, that are uniform over $f\in L_2[0,1)^d$ and depend on the point set being scrambled.

With this notation, the column-wise mean dimension for scrambled nets is
$$
\nu_{\net}(F) = \frac{\sum_{u\subseteq1:sK}|u|\sigma^2_u(f)\tilde\Gamma_u}
{\sum_{u\subseteq1:sK}\sigma^2_u(f)\tilde\Gamma_u}.
$$
We expect smaller $\tilde\Gamma_u$ for small $|u|$ and downweighting the small cardinalities has the effect of raising $\nu(F)$. We then expect $\nu_{\net}(F)>\nu(f)$.  We can understand an RQMC-friendly integrand as one with a small value of $\nu(f)$. 
It may very well have a large value of $\nu(F)$ if most of the low dimensional variation is integrated without error.

For scrambled Sobol' points, $f_0$ has the components of a Fourier-Walsh decomposition that the nets integrate without error and $f_1$ has other components that contribute uncorrelated errors that each have a variance $O(1/n)$.  The constants within $O(1/n)$ vary with the Walsh functions. For a discussion of those gain coefficients in scrambled Sobol' points see \cite{gainpowtwo}.  The column-wise mean dimension for scrambled nets is then the mean dimension of $f_1$ after weighting each Walsh function by the square root of a gain coefficient.

In summary, we think of $\nu(F)$ as the mean dimension of the hard to integrate parts of $f$ for whatever method we are using, up to gain factor weights where necessary.  For plain MC the hard to integrate parts are all except the constant term and so $\nu_{\mc}(F)=\nu_{\mc}(f)$.

To understand the dimensional effect of Array-RQMC-WoS we can compute its column-wise mean dimension. Just as RQMC-WoS is compared to MC-WoS, we compare the mean dimension of Array-RQMC-WoS to an analogous Array-MC-WoS. An Array-MC-WoS algorithm chooses angles at each step independently of the points' location on the Hilbert curve. Array-MC-WoS and MC-WoS have the same accuracy as we see next, but they have different mean dimensions. 

\subsection{Array-MC-WoS}

\newcommand{\bszikmc}{\bsz_{ik}^{\mathrm{MC}}}
\newcommand{\bszikamc}{\bsz_{ik}^{\mathrm{AMC}}}
\newcommand{\allzkmc}{\mathcal{Z}_{k}^{\mathrm{MC}}}
\newcommand{\allzkamc}{\mathcal{Z}_{k}^{\mathrm{AMC}}}
\newcommand{\allxkmc}{\mathcal{X}_{k}^{\mathrm{MC}}}
\newcommand{\allxkamc}{\mathcal{X}_{k}^{\mathrm{AMC}}}
\newcommand{\allxknextmc}{\mathcal{X}_{k+1}^{\mathrm{MC}}}
\newcommand{\allxknextamc}{\mathcal{X}_{k+1}^{\mathrm{AMC}}}
\newcommand{\allzknextamc}{\mathcal{Z}_{k+1}^{\mathrm{AMC}}}
\newcommand{\allzknextmc}{\mathcal{Z}_{k+1}^{\mathrm{MC}}}
\newcommand{\allzkinitmc}{\mathcal{Z}_{0}^{\mathrm{MC}}}
\newcommand{\allzkinitamc}{\mathcal{Z}_{0}^{\mathrm{AMC}}}
\newcommand{\allz}{\mathcal{Z}}

We contrast Array-MC-WoS with MC-WoS by denoting the positions of walker $i$ at step $k$ under these algorithms by $\bszikamc\in\Omega$ and $\bszikmc\in\Omega$ respectively. In MC-WoS, the state of all $n$ walkers at step $k$ is given by $\allzkmc\in\real^{n\times d}$ whose $i$th row is $\bszikmc$. The analogous quantity for Array-MC-WoS is $\allzkamc\in\real^{n\times d}$. First, we show that these matrices have the same distribution.

\begin{prop}

For every $k\ge0$, $\allzkmc\eqd\allzkamc$.

\end{prop}

\begin{proof}
If $k\in\{0,1\}$ then $\allzkmc=\allzkamc$ (since we do not sort the $\bsz_{i0}$) which implies that $\allzkmc\eqd\allzkamc$ for those $k$. We will proceed by induction showing that if $\allzkmc\eqd\allzkamc$ then $\allzknextmc\eqd\allzknextamc$. 

For $i=1,\dots,n$, let $\bsx_{i,k+1}\simiid\runif([0,1)^s)$ be the vectors used to advance walker $i$ from $\bsz_{ik}$ to $\bsz_{i,k+1}$ in MC-WoS and let $\allxknextmc\in\real^{n\times s}$ have $i$th row $\bsx_{i,k+1}$. 
For plain MC-WoS, walker $i$ evolves according to
$\bsz_{i,k+1}^{\mathrm{MC}} = \phi\bigl(\bsz_{ik}^{\mathrm{MC}}, \bsx_{i,k+1}\bigr)$
where $\phi$ is the WoS update map. The ensemble then evolves according to $\allzknextmc =\Phi( \allzkmc,\allxknextmc)$ where $\Phi$ applies $\phi$ row by row. The Array-MC-WoS ensemble evolves as $\allzknextamc=\Phi(\allzkamc,\allxknextamc)$ where $\allxknextamc$ has the same rows as $\allxknextmc$ in a permuted order. 

That permutation is deterministic conditionally on $\allzkamc$ and so $\allxknextamc$ has the same conditional distribution as $\allxknextmc$.  Then for measurable $A\subset[0,1)^{n\times d}$
\begin{align*}
\Pr( \allzknextamc\in A) &=
\e\left( \Pr\bigl( \Phi(\allzkamc,\allxknextamc)\in A\giv \allzkamc\bigr)\right)\\
&=
\e\left( \Pr\bigl( \Phi(\allzkamc,\allxknextmc)\in A\giv \allzkamc\bigr)\right)\\
&=\Pr\bigl( \Phi(\allzkamc,\allxknextmc)\in A\bigr).
\end{align*}
Now we assume that $\allzkamc\eqd\allzkmc$.
Then $\Pr( \allzknextamc\in A)=\Pr(\allzknextmc\in A)$
and so $\allzknextamc\eqd\allzknextmc$,  completing the induction.
\end{proof}  

Because $\allzkamc\eqd\allzkmc$, a fixed $k$ version of Array-MC-WoS has the same distribution and MSE as MC-WoS.  What is not obvious at first, is that these algorithms can have different column-wise mean dimensions. We illustrate the difference for a problem with $\Omega\subset\real^2$ and no source term. 

For some $x\in[0,1)$ a walker moves in direction $\theta(x)=(\cos(2\pi x),\sin(2\pi x))^\tran$. The position of all $n$ walkers is given by $\allz_k \in\real^{n\times 2}$ with $i$th row $\bsz_{ik}$. We use $\vec\bsx_k = (x_{1k},\dots,x_{nk})^\tran\in[0,1)^n$ to represent all $n$ directions used to advance the walkers to step $k$.  Walker $i$ gets component $\varpi_i(k)$ of $\vec\bsx_{k}$ where $\varpi(k)$ is the permutation of $(1,2,\dots,n)$ from the Hilbert sort, so $\varpi(k)$ depends on $\vec\bsx_1,\dots,\vec\bsx_{k-1}$.

The first step from $\allz_0$ yields $\allz_1$ identically for MC-WoS and Array-MC-WoS.  For step $2$ under Array-MC-WoS
\begin{align*}
\bsz_{i2} &= \bsz_{i1} + r_{i2}\theta(x_{\varpi_i(2),2}),\quad\text{so}\\
\allz_2(\vec\bsx_1,\vec\bsx_2) &=
\allz_1(\vec\bsx_1) + \diag(\vec \bsr_2(\allz_1(\vec{\bsx}_1)))
\Theta_2(\vec\bsx_1,\vec\bsx_2)
\end{align*}
where $\vec\bsr_k$ has all $n$ radii used at step $k$ and
$\Theta_2\in[-1,1]^{n\times 2}$ has the $i$th row $\theta(x_{\varpi_i(2),2})$. Here $\Theta_2$ depends on $\vec\bsx_1$ through $\varpi(2)$. For plain MC-WoS, $\Theta_2$ only depends on $\vec\bsx_2$.  This distinction continues to later steps where $\Theta_k$ depends on $\vec\bsx_1,\dots,\vec\bsx_k$ under Array-MC-WoS but only depends on $\vec\bsx_k$ under MC-WoS.  If $k<k'$, then under Array-MC-WoS the vector $\vec{\bsx}_k$ can affect more subsequent transitions than $\vec{\bsx}_{k'}$. In contrast, under MC-WoS, each $\vec{\bsx}_k$ affects only the $k$th transition, so it makes sense that those algorithms can have different mean dimensions.

We computed vector-wise Sobol’ indices for the gasket example for $6$ methods: MC-WoS, Sobol-WoS, Lattice-WoS, Array-MC-WoS, Array-Sobol-WoS and Array-Lattice-WoS. The walks used the usual $\varepsilon$-stopping rule, and the Sobol' analysis was applied to the first $20$ step-vectors, giving a partial estimate of the mean dimension. The corresponding mean dimensions are reported in Table \ref{tab:mean-dimensions}. We observe that moving from plain MC to plain RQMC increases the mean dimension, suggesting that RQMC removes part of the lower-dimensional variance. The mean dimension increases still further for Array-RQMC, which indicates that even more of the lower-dimensional variance is removed, leaving a more high-dimensional residual variance. We see that the lattice sampling leaves behind a residual with a higher column-wise mean dimension than the Sobol' points. This is consistent with the Array-Lattice-WoS having somewhat better accuracy than Array-Sobol-WoS in Table~\ref{tab:mserf-summaries}.

\begin{table}[t]
\centering
\begin{tabular}{lcc}
\toprule
Base algorithm & $\hat\nu(F)$ non-array & $\hat\nu(F)$ array \\
\midrule
MC-WoS      & $\phz3.24$ & $\phz2.13$ \\
Sobol-WoS   & $\phz7.43$ & $12.58$    \\
Lattice-WoS & $\phz8.16$ & $13.81$    \\
\bottomrule
\end{tabular}
\caption{
Estimated vector-wise mean dimensions of the WoS estimators for the gasket example with $n=4096$ using $800$ replicates.
\label{tab:mean-dimensions}}
\end{table}

The Sobol' indices that contribute to mean dimension are shown in Figure~\ref{fig:sobol_indices}. The left panel plots normalized indices $\olt^2_k/\sigma^2$ versus $k$. Over 90\% of the variance in MC comes from effects that include the first step variable $\vec{\bsx}_1$ while the comparable number for $\vec\bsx_{10}$ is just over 10\%. The normalized indices decay more slowly for RQMC and even more slowly for Array-RQMC.  The Array-RQMC value of $\olt^2_{20}$ is still quite large so it is clear that increasing $k$ would increase $\nu(F)$ for a $k$-step algorithm. We see a difference between MC and Array-MC for small $k$, especially $k=1$. The greater importance of step 1 in Array-MC-WoS explains its lower column-wise mean dimension in Table~\ref{tab:mean-dimensions}.

The right panel of Figure~\ref{fig:sobol_indices} shows that the unnormalized Sobol' indices $\olt^2_k(F)$ generally decrease as we replace MC-WoS by RQMC-WoS and then decrease again as we replace RQMC-WoS by Array-RQMC-WoS. The reduction factors are largest for small $k$ and close to one for large $k$.

\begin{figure}
    \centering
    \includegraphics[width=1\linewidth]{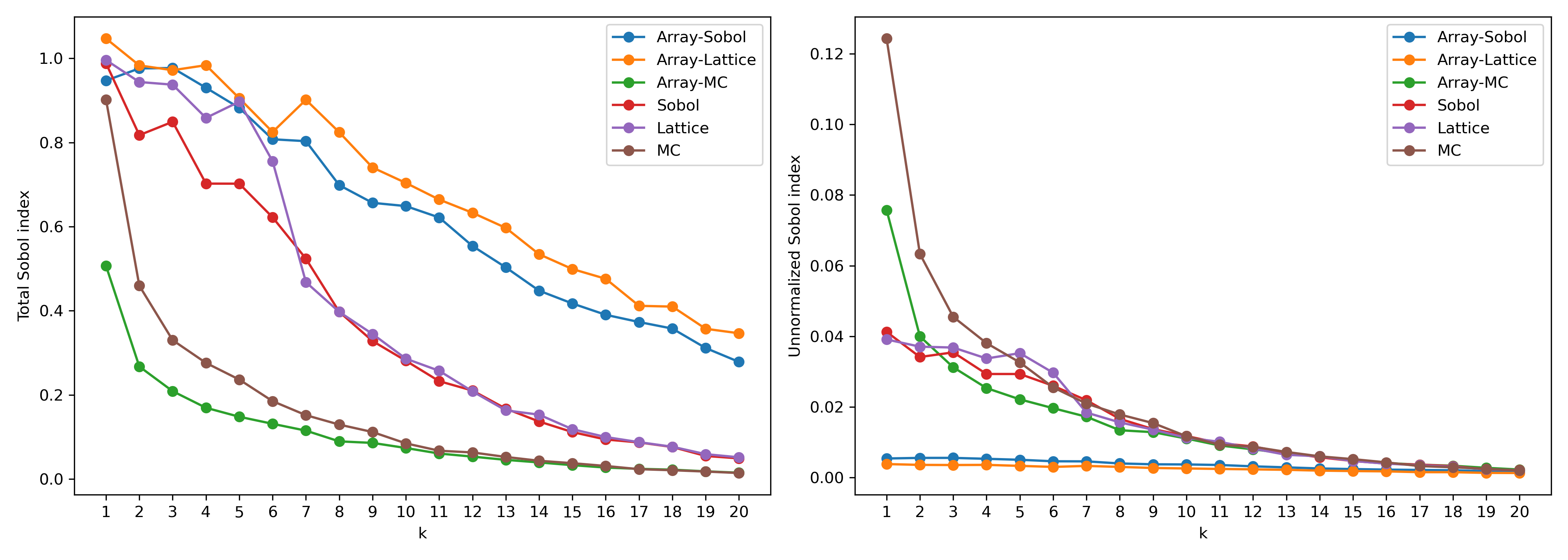}
    \caption{Normalized and unnormalized total Sobol' indices for the gasket example, with $\bsz_0=\bsz_{*}$, for trajectories of size $n=4096$, computed over $800$ replicates for $k=1,\dots,20$.
    }
    \label{fig:sobol_indices}
\end{figure}

\section{Discussion}\label{sec:discussion}

We have found that using Array-RQMC in WoS improves accuracy by some very large factors, in some instances reducing the MSE by more than 3000-fold.  The theoretical understanding of Array-RQMC is not yet advanced enough to explain this improvement.  There is some recent theoretical study of RQMC-WoS in \cite{rqmcwos:tr} giving variance rates of $O(n^{-1-1/k})$ for some integrands at step $k$, but that method does not perform nearly as well as Array-RQMC-WoS in our examples.

One interesting finding was that Array-Lattice-WoS consistently outperformed Array-Sobol-WoS, although the improvement was by modest factors.  Perhaps the lattice methods benefit from the periodic nature of the WoS integrands. Those integrands are not smooth enough to give lattices a better RQMC rate, though possibly that could explain a better constant.  On the other hand, Lattice-WoS does not appear to have an advantage over Sobol-WoS.

Another finding was that the variance or MSE versus $n$ generally showed mild nonlinearity in log-log plots, as if the slope were improving. It would be very interesting to find an explanation for that behavior.

\section*{Acknowledgments}

We thank Rohan Sawhney, Bob Carpenter,  Charlie Epstein and Frances Kuo for helpful comments.

\bibliography{array_wos}
\bibliographystyle{chicago}
\end{document}